\documentclass[a4paper,11pt]{article}

\author{Sebastian Pancratz, Jan Tuitman}
\title{Improvements to the deformation method for counting points 
on smooth projective hypersurfaces}

\usepackage[hmargin=3.2cm,vmargin=3.2cm,a4paper,centering,twoside]{geometry}

\usepackage[T1]{fontenc}
\usepackage{ae,aecompl}
\usepackage{verbatim}
\usepackage{ifpdf}
\usepackage{booktabs}

\usepackage{hyperref}
\hypersetup{
    colorlinks=false,   
    citecolor=green,    
    filecolor=magenta,  
    linkcolor=red,      
    urlcolor=blue       
}

\makeatletter
\newcommand\org@hypertarget{}
\let\org@hypertarget\hypertarget
\renewcommand\hypertarget[2]{%
    \Hy@raisedlink{\org@hypertarget{#1}{}}#2%
} 
\makeatother

\ifpdf
    \hypersetup{
        pdftitle={Deformation method},
        pdfauthor={Sebastian Pancratz, Jan Tuitman},
        pdfsubject={Computational Number Theory},
        bookmarks=true,
        bookmarksnumbered=true,
        unicode=true,
        pdfstartview={FitH},
        pdfpagemode={UseOutlines}
    }
\fi

\usepackage[section]{algorithm}
\usepackage[noend]{algpseudocode}

\usepackage{natbib}

\bibpunct{[}{]}{,}{n}{}{}

\usepackage{url}

\makeatletter
\def\url@leostyle{%
  \@ifundefined{selectfont}{\def\UrlFont{\sf}}{\def\UrlFont{\small\ttfamily}}}
\makeatother
\usepackage{paralist}

\usepackage{amsmath,amsthm,amscd,amsfonts,amssymb}
\usepackage{cases}
\usepackage[all]{xy}

\allowdisplaybreaks[4]
\numberwithin{equation}{section}

\providecommand{\card}[1]{\lvert#1\rvert}                

\providecommand{\abs}[1]{\lvert#1\rvert}                 

\providecommand{\floor}[1]{\left\lfloor#1\right\rfloor}   
\providecommand{\floorbig}[1]{\bigl\lfloor#1\bigr\rfloor} 

\providecommand{\ceil}[1]{\left\lceil#1\right\rceil}   

\newcommand{\NN}{\mathbf{N}} 
\newcommand{\ZZ}{\mathbf{Z}} 
\newcommand{\QQ}{\mathbf{Q}} 
\newcommand{\RR}{\mathbf{R}} 
\newcommand{\CC}{\mathbf{C}} 
\newcommand{\FF}{\mathbf{F}} 

\renewcommand{\to}{\rightarrow}        

\DeclareMathOperator{\fIm}{im}       

\DeclareMathOperator{\ord}{ord}          
\DeclareMathOperator{\Frob}{F}           
\DeclareMathOperator{\Spec}{Spec}        

\providecommand{\HdR}{H_{\text{dR}}}    
\providecommand{\Hrig}{H_{\text{rig}}}  

\providecommand{\cB}{\mathcal{B}} 

\providecommand{\BigOh}{O}          
\providecommand{\SoftOh}{\tilde{O}} 

\theoremstyle{definition}

\newtheorem{thm}{Theorem}[section]
\newtheorem{lem}[thm]{Lemma}
\newtheorem{prop}[thm]{Proposition}
\newtheorem{cor}[thm]{Corollary}
\newtheorem{defn}[thm]{Definition}

\newtheorem{rem}[thm]{Remark}

\newtheorem{assump}[thm]{Assumption}

\makeatletter

\newcommand{\Rmnum}[1]{\expandafter\@slowromancap\romannumeral #1@}
\makeatother

\begin{document}

\maketitle

\abstract{We present various improvements to the deformation method
for computing the zeta function of smooth projective hypersurfaces 
over finite fields using $p$-adic cohomology.  This includes new bounds 
for the $p$-adic and $t$-adic precisions required to obtain provably 
correct results and gains in the efficiency of the individual steps of 
the method.  The algorithm that we thus obtain has lower time and space 
complexities than existing methods.  Moreover, our implementation is 
more practical and can be applied more generally, which we illustrate 
with examples of generic quintic curves and quartic surfaces.}

\bigskip 

\setcounter{tocdepth}{2}
\tableofcontents

\bigskip


\section{Introduction}
\label{sec:Introduction}

Let $\FF_q$ denote a finite field with $q$ elements, where $q$ is a power of 
the prime number~$p$, and let $X$ denote an algebraic variety over~$\FF_q$. 

\begin{defn}
The zeta function of $X$ is the formal power series
\[
Z(X,T) = \exp \Bigl(\sum_{i=1}^{\infty} \card{X(\FF_{q^i})} \frac{T^i}{i} \Bigr).
\]
\end{defn}

As we will see in the next section, $Z(X,T)$ 
is a rational function, i.e.\ it is contained in $\QQ(T)$, and 
hence can be given by a finite amount of data. Therefore, it is natural 
to ask whether it can be computed effectively and, in fact, it is not 
hard to provide an algorithm as follows. Using well known bounds 
by Bombieri~\citep{Bombieri1966} for the degrees of the numerator and 
denominator of $Z(X,T)$, one can reduce the computation of $Z(X,T)$ to that 
of a finite number of the $\card{X(\FF_{q^i})}$, which can be determined 
by naive counting.

A more interesting problem is whether, and if so how, $Z(X,T)$ can be 
computed efficiently, where `efficiently' can mean with low time complexity 
or just fast in practice. When $X$ is a (hyper-)elliptic curve, this problem
is important in cryptographic applications and has been the subject of 
much attention, resulting in very efficient algorithms.  For example, when 
$X$ is an elliptic curve, Schoof's algorithm~\citep{Schoof1995}, which uses 
$\ell$-adic \'etale cohomology, has runtime polynomial in~$\log(q)$, and is 
also very fast in practice using improvements due to Atkins and Elkies.

For more general algebraic varieties~$X$, the only available option is 
usually to compute the rigid cohomology spaces of~$X$, with their natural 
action of the Frobenius map, and then use a Lefschetz formula to deduce the 
zeta function. This method was introduced by Kedlaya in the case of 
hyperelliptic curves in odd characteristic~\citep{Kedlaya2001}.  The same 
idea has been shown to work in much greater generality, for example for 
smooth projective hypersurfaces \citep{AbbottKedlayaRoe2006}. 

Lauder \citep{Lauder2004a,Lauder2004b} showed that instead of computing the 
action of the Frobenius map on the rigid cohomology spaces of a smooth 
projective hypersurface~$X$ directly, it is better, at least in terms of 
time complexity, to embed $X$ in a family of smooth projective hypersurfaces 
containing a diagonal hypersurface.  Following his deformation method, 
one first computes the action of the Frobenius map on the rigid cohomology 
of the diagonal hypersurface, and then solves a $p$-adic differential equation 
to obtain the Frobenius map on the rigid cohomology of the original 
hypersurface. 

To be more precise, in~\citep{Lauder2004a,Lauder2004b} Lauder did not directly 
work with rigid cohomology but with Dwork cohomology. While the two 
cohomology theories are equivalent, various comparison and finiteness results 
are more easily stated and proved in the context of rigid cohomology.  In~\citep{Gerkmann2007} 
Gerkmann reformulated Lauder's deformation method in terms of rigid 
cohomology.  Moreover, he improved various precision bounds, making 
the algorithm more practical, which he demonstrated with many examples. Kedlaya 
introduced new ideas and results to further lower the precision bounds for 
the deformation method in \citep{Kedlaya2012}.

The aim of this paper is to continue where Lauder, Gerkmann and Kedlaya left 
off. We make improvements to almost every step of the algorithm. This results 
in an algorithm with both lower time and space complexity than Lauder's 
original algorithm, but perhaps more importantly, which is a lot more efficient
in practice. The first author has written a (publicly available) 
implementation of our algorithm using the library FLINT~\citep{FLINT}. 
This implementation lowers the runtimes of the examples in~\citep{Gerkmann2007} 
by factors of $50$ to $5,000$. Moreover, it can be used to compute the zeta 
function in many cases where this was not possible before, e.g.\ for generic 
quartic surfaces.

We now briefly describe the contents of the remaining sections,
necessarily relying on some terminology that is introduced only 
in Section~\ref{sec:Background}.  The reader who is not familiar 
with this terminology might prefer to start reading there.

In Section~\ref{sec:Background}, we recall the main theoretical results that 
underpin the remaining sections of the paper.  We limit ourselves to the bare 
minimum as there are already good references available for the relevant theory, 
see e.g.~\citep{Kedlaya2012}. We also introduce the required terminology and 
notation. We conclude the section with an overview of the different steps of 
the deformation method, which are then treated individually in the next four 
sections.

In Section~\ref{sec:Connection}, we explain how to compute the Gauss--Manin 
connection on the cohomology of a family of smooth projective hypersurfaces. 
We define an explicit monomial basis for the cohomology that we will use 
throughout the paper. Our most important result in this section
is Theorem~\ref{thm:Isomorphism}, which allows us to compute very efficiently 
in the cohomology. We formalise the computation of the Gauss--Manin connection 
matrix in Algorithm~\ref{alg:Connection}. We also prove some lower bounds for 
the valuation of the matrix of Frobenius and its inverse that are essential 
for controlling the $p$-adic precision loss in the algorithm.

In Section~\ref{sec:Diagonal}, we show how to compute the Frobenius matrix of
a diagonal hypersurface over a prime field. Our method is essentially based 
on a computation of Dwork, but by rewriting and slightly generalising his formulas 
we obtain Algorithm~\ref{alg:Diagfrob}, which is a significant improvement to the 
corresponding algorithms of Lauder and Gerkmann, both in terms of time complexity 
and in practice.

In Section~\ref{sec:DifferentialSystem}, we explain how to solve the 
differential equation for the Frobenius matrix. We use the same method as 
Lauder but incorporate improved convergence bounds for $p$-adic differential 
equations by Kedlaya. We collect the precision bounds that follow from our 
analysis in Theorem~\ref{thm:Ni} and formalise the computation of the power 
series expansion of the Frobenius matrix in Algorithm~\ref{alg:expansion}.

In Section~\ref{sec:ZetaFunctions}, we describe how to evaluate the Frobenius 
matrix at some fibre and compute its zeta function. We combine various bounds 
from different sources to lower the required $p$-adic and $t$-adic precisions. 
This finally results in Algorithm~\ref{alg:complete}, which combines all of 
our previous algorithms, and is the main result of the paper.

In Section~\ref{sec:Complexity}, we analyse the time and space complexity of 
our algorithm and compare these to Lauder's work~\citep{Lauder2004a}. 
In Section~\ref{sec:Examples}, we compute various numerical examples, and 
compare our runtimes to those provided by Gerkmann~\citep{Gerkmann2007}. 

Both authors were supported by the European Research Council (grant 204083)
and additionally the second author was supported by FWO - Vlaanderen. We would 
like to thank Alan Lauder for all his help and in particular for his comments and
suggestions on earlier versions of this paper. Finally, we thank the anonymous
referees for their comments and suggestions.

\section{Theoretical background}
\label{sec:Background}

We start by recalling the main result about the zeta function of algebraic
varieties over finite fields.

\begin{thm}[Weil conjectures] If $X/\FF_q$ is a smooth projective variety of 
dimension~$m$, then \label{thm:weildeligne}
\[
Z(X,T)=\frac{p_1 p_3 \dotsm p_{2m-1}}{p_0 p_2 p_4 \dotsm p_{2m}},
\]
where for all $i$:
\begin{enumerate}
\item $p_i = \prod_j (1-\alpha_{i,j}T) \in \ZZ[T]$, 
\item the transformation $t \rightarrow q^m/t$ maps the $\alpha_{i,j}$ 
      bijectively to the $\alpha_{2m-i,k}$, preserving multiplicities,
\item $\abs{\alpha_{i,j}} = q^{i/2}$ for all $j$, and every embedding 
      $\bar{\QQ} \hookrightarrow \CC$. 
\end{enumerate}
\end{thm}

\begin{proof}
The proof of this theorem was completed by Deligne in \citep{Deligne1974}.
\end{proof}

We let $\QQ_q$ denote the unique unramified extension of $\QQ_p$ with 
residue field $\FF_q$ and $\ZZ_q$ its ring of integers. We denote
the $p$-adic valuation on $\QQ_q$ by $\ord_p(-)$. 

\begin{defn}
Let $\Hrig^{i}(X)$ denote the rigid
cohomology spaces of $X$. These are finite dimensional vector spaces 
over $\QQ_q$ that are contravariantly functorial in $X$, and they are 
equipped with an action of the $p$-th and $q$-th power Frobenius map 
on $X$ that we denote by $\Frob_p$ and $\Frob_q$, respectively. For the 
construction and basic properties of these spaces we refer 
to~\citep{Berthelot1986}.
\end{defn}

The relation between the zeta function and the rigid cohomology spaces 
is given by the so called Lefschetz formula.

\begin{thm}[Lefschetz formula] \label{thm:Lefschetz}
If $X$ is a smooth proper algebraic variety over~$\FF_q$ of dimension~$m$, 
then 
\[
Z(X,T) = \prod_{i=0}^{2m} \det \bigl(1- T \Frob_q | \Hrig^i(X) \bigr)^{(-1)^{i+1}}.
\]
\end{thm}

\begin{proof}
See for example \citep[Theorem 6.3]{EtesseLeStum1993}.
\end{proof}

Let $\pi \colon \mathfrak{X} \rightarrow \mathfrak{S}$ be a smooth family of 
algebraic varieties defined over~$\QQ_q$.

\begin{defn}
Let $\HdR^i(\mathfrak{X}/\mathfrak{S})$ denote the $i$-th relative algebraic 
de Rham cohomology sheaf on $\mathfrak{S}$. 
If $\mathfrak{X}/\mathfrak{S}$ admits a relative normal crossing 
compactification, then the $\HdR^i(\mathfrak{X}/\mathfrak{S})$ are vector bundles.
\end{defn}

The $\HdR^i(\mathfrak{X}/\mathfrak{S})$ come equipped with an integrable 
connection, which is called the Gauss--Manin connection. Let us first recall 
the notion of a connection on a vector bundle.

\begin{defn}
Let $\mathfrak{E}$ be a vector bundle on $\mathfrak{S}$. A connection on 
$\mathfrak{E}$ is a map of vector bundles 
$\nabla \colon \mathfrak{E} \rightarrow \Omega^1_{\mathfrak{S}} \otimes \mathfrak{E}$
which satisfies the Leibniz rule
\begin{align*}
\nabla(f e)=f\nabla(e)+df \otimes e
\end{align*} 
for all local sections~$f$ of $\mathcal{O}_{\mathfrak{S}}$ and $e$ 
of $\mathfrak{E}$.
\end{defn}

The Gauss--Manin connection on $\HdR^i(\mathfrak{X}/\mathfrak{S})$ can 
be defined as follows.

\begin{defn}
The de Rham complex $\Omega^{\bullet}_{\mathfrak{X}}$ can be equipped 
with the decreasing filtration
\[
F^i=\fIm(\Omega^{\bullet-i}_{\mathfrak{X}} \otimes \pi^* \Omega^i_{\mathfrak{S}} \rightarrow \Omega^{\bullet}_{\mathfrak{X}}). 
\]
The spectral sequence associated to this filtration has as its first sheet 
\[
E_1^{p,q}=\Omega^p_{\mathfrak{S}} \otimes \HdR^q(\mathfrak{X}/\mathfrak{S}).
\]
The Gauss--Manin connection 
$\nabla:H^i(\mathfrak{X}/\mathfrak{S}) \rightarrow \Omega^1_{\mathfrak{S}} \otimes H^i(\mathfrak{X}/\mathfrak{S})$ 
is now defined as the differential $d_1 \colon E_1^{0,i} \rightarrow E_1^{1,i}$ 
in this spectral sequence.
\end{defn}

\begin{rem}
We can give a more explicit description of $\nabla$ when 
$\mathfrak{X}/\mathfrak{S}$ is affine. If we lift a relative $i$-cocycle 
$\omega \in \Omega^i_{\mathfrak{X}/\mathfrak{S}}$ to an absolute $i$-form 
$\omega' \in \Omega^i_{\mathfrak{X}}$ and apply the absolute differential~$d$, 
we get an element of 
$\Omega^1_{\mathfrak{S}} \wedge \Omega^i_{\mathfrak{X}/\mathfrak{S}}$. 
Projecting onto 
$\Omega^1_{\mathfrak{S}} \otimes \HdR^i(\mathfrak{X}/\mathfrak{S})$, 
we obtain $\nabla(\omega)$. 
\end{rem}

\begin{defn} \label{defn:sigma}
We write $\sigma$ for the standard $p$-th power Frobenius lift on 
$\mathbf{P}^1_{\QQ_q}$, i.e.\ the semilinear map that lifts the $p$-th power 
Frobenius map on $\mathbf{P}^1_{\FF_q}$ and satisfies $\sigma(t)=t^p$. 
\end{defn}

Now suppose that $\mathfrak{E}$ is a vector bundle with connection on 
some Zariski open subset~$\mathfrak{S}$ of $\mathbf{P}^1_{\QQ_q}$ with 
complement~$\mathfrak{Z}$. Let~$V$ denote the rigid analytic subspace 
of~$\mathbf{P}^1_{\QQ_q}$ which is the complement of the union of the 
open disks of radius~$1$ around the points of~$\mathfrak{Z}$.

\begin{defn}
A Frobenius structure on $\mathfrak{E}$ is an isomorphism of vector bundles 
with connection $F \colon \sigma^* \mathcal{E} \rightarrow \mathcal{E}$ 
defined on some strict neighbourhood of $V$. 
\end{defn}

\begin{thm} \label{thm:frobstruc}
Let $\mathcal{S}$ be a Zariski open subset of $\mathbf{P}^1_{\ZZ_q}$ and 
suppose that $\mathcal{X}/\mathcal{S}$ is a smooth family of algebraic 
varieties that admits a relative normal crossing compactification. Denote 
the generic fibres of $\mathcal{S}$, $\mathcal{X}$ by 
$\mathfrak{S}=\mathcal{S} \otimes \QQ_q$, $\mathfrak{X}=\mathcal{X} \otimes \QQ_q$ 
and the special fibres by $S=\mathcal{S} \otimes \FF_q$, 
$X=\mathcal{X} \otimes \FF_q$, respectively. The vector bundle 
$\HdR^i(\mathfrak{X}/\mathfrak{S})$ with the Gauss--Manin connection $\nabla$ 
admits a Frobenius structure $F$ with the following property. 
For any finite field extension $\FF_{\mathfrak{q}}/\FF_q$ and
all $\tau \in S(\FF_{\mathfrak{q}})$,
\[
(\Hrig^i(X_{\tau}),\Frob_p) \cong (\HdR^i(\mathfrak{X}/\mathfrak{S}),F)_{\hat{\tau}}
\] 
as $\QQ_{\mathfrak{q}}$-vector spaces with a $\sigma$-semilinear 
endomorphism, where $\hat{\tau} \in \mathcal{S}(\ZZ_{\mathfrak{q}})$ denotes 
the Teichm\"uller lift of $\tau$. We will therefore denote this Frobenius 
structure on $\HdR^i(\mathfrak{X}/\mathfrak{S})$ by $\Frob_p$ as well.
\end{thm}

\begin{proof}
This result is well known, see for example \citep[Theorem 6.1.3]{Kedlaya2012}.
Although it is usually attributed to Berthelot, a complete reference 
seems to be missing from the literature.
\end{proof}

\begin{defn}
Let $\Hrig^i(X/S)$ denote the vector bundle $\HdR^i(\mathfrak{X}/\mathfrak{S})$ 
with its Frobenius structure $\Frob_p$ from Theorem~\ref{thm:frobstruc}.
\end{defn}

\begin{rem}
One can show that $\Hrig^i(X/S)$ is again functorial in $X/S$ and so does 
not depend on the lift $\mathcal{X}/\mathcal{S}$. Moreover, one can still 
define $\Hrig^i(X/S)$ when $X/S$ cannot be lifted to characteristic zero, 
see~\citep{Berthelot1986}.  However, for our purposes the above definition 
will be sufficient.
\end{rem}

In this paper we restrict our attention to one-parameter families of smooth 
projective hypersurfaces. So we let $P \in \ZZ_q[t][x_0,\dotsc,x_n]$ denote 
a homogeneous polynomial of degree $d$ and let 
$\mathcal{S} \subset \mathbf{P}^1_{\ZZ_q}$ be a Zariski open subset such that 
$P$ defines a family $\mathcal{X}/\mathcal{S}$ of smooth hypersurfaces contained 
in $\mathbf{P}^n_{\mathcal{S}}$. 
We let $\mathcal{U}/\mathcal{S}$ 
denote the complement of $\mathcal{X}/\mathcal{S}$ in $\mathbf{P}^n_{\mathcal{S}}$,
and write $\mathfrak{X}=\mathcal{X} \otimes \QQ_q$, 
$\mathfrak{U}=\mathcal{U} \otimes \QQ_q$, $\mathfrak{S}=\mathcal{S} \otimes \QQ_q$
for the generic fibres, and $X=\mathcal{X} \otimes \FF_q$, 
$U=\mathcal{U} \otimes \QQ_q$, $S=\mathcal{S} \otimes \FF_q$ for the special 
fibres of $\mathcal{X},\mathcal{U}$, $\mathcal{S}$, respectively. Moreover, 
we let $\FF_{\mathfrak{q}}/\FF_q$ denote a finite field extension and 
denote $a=\log_p(\mathfrak{q})$.

\begin{thm} \label{thm:hypersurface} 
For all $\tau \in S(\FF_{\mathfrak{q}})$, we have
\begin{equation} \label{eq:formulazeta}
Z(X_{\tau},T) = \frac{\chi(T)^{(-1)^n}}{(1 - T) (1 - \mathfrak{q}T) \dotsm (1 - \mathfrak{q}^{n-1}T)},
\end{equation}
where 
$\chi(T) = \det \bigl( 1 - T \mathfrak{q}^{-1} \Frob_{\mathfrak{q}} | \Hrig^n(U_{\tau}) \bigr) \in \ZZ[T]$
denotes the reverse characteristic polynomial of the action of 
$\mathfrak{q}^{-1} \Frob_{\mathfrak{q}}$ on $\Hrig^n(U_{\tau})$. 
Moreover, the polynomial $\chi(T)$ has degree 
\begin{equation} \label{eq:formulab}
\frac{1}{d} \bigl((d-1)^{n+1} + (-1)^{n+1}(d-1) \bigr).
\end{equation}
\end{thm}

\begin{proof}
This theorem is well known, see for example~\citep{AbbottKedlayaRoe2006}. 
Since we need some intermediate results from the proof later on, we will 
give a brief sketch here. First, by Theorem~\ref{thm:Lefschetz}, we have 
\[
Z(X_{\tau},T) = \prod_{i=0}^{2(n-1)} 
    \det\bigl(1- T \Frob_{\mathfrak{q}} | \Hrig^i(X_{\tau})\bigr)^{(-1)^{i+1}}.
\]
Then, by the Lefschetz hyperplane theorem and Poincar\'e duality, we see 
that $\Hrig^i(X_{\tau}) \cong \Hrig^i(\mathbf{P}^n_{\FF_{\mathfrak{q}}})$ 
for all $i \neq (n-1)$. Next, one shows by a computation that
\begin{equation} \label{eqn:projcoho}
\Hrig^i(\mathbf{P}^n_{\FF_{\mathfrak{q}}}) 
\cong 
\begin{cases}
\QQ_{\mathfrak{q}}(i) &\mbox{if $i$ even,} \\
$0$ &\mbox{if $i$ odd,} 
\end{cases} 
\end{equation}
where $(i)$ denotes the $i$-th Tate twist, for which $\Frob_p$ is multiplied 
by $p^{-i}$. It remains to determine $\Hrig^{n-1}(X_{\tau})$. One uses the 
excision short exact sequence
\begin{equation} \label{eqn:excision}
\begin{CD}
0 @>>> \Hrig^{n}(U_{\tau}) @>>> \Hrig^{n-1}(X_{\tau})(-1) @>>> \Hrig^{n+1}(\mathbf{P}^n_{\FF_{\mathfrak{q}}}) @>>> 0
\end{CD} 
\end{equation}
to relate $\Hrig^{n-1}(X_{\tau})$ to $\Hrig^{n}(U_{\tau})$ and complete 
the proof of~\eqref{eq:formulazeta}. We will show in Proposition~\ref{prop:rankcoho}
that the dimension of $\Hrig^{n-1}(U_{\tau})$ is given by~\eqref{eq:formulab}.
\end{proof}

Let $[e_1, \dotsc, e_b]$ be some basis of sections of 
$\HdR^n(\mathfrak{U}/\mathfrak{S})$, and let $M \in M_{b \times b}(\QQ_q(t))$ 
denote the matrix of the Gauss--Manin connection $\nabla$ with respect 
to this basis, i.e.
\[
\nabla (e_j) = \sum_{i=1}^b M_{i,j} e_i.
\]
Let $r \in \ZZ_q[t]$ with $\ord_p(r)=0$ be a denominator for $M$, i.e.\ such that we can write 
$M = G/r$ with $G \in M_{b \times b}(\QQ_q[t])$, and let $\Phi$ denote the 
matrix of $p^{-1}\Frob_p$ with respect to the basis $[e_1, \dotsc, e_b]$, i.e.\
\[
p^{-1} \Frob_p (e_j) = \sum_{i=1}^b \Phi_{i,j} e_i.
\]

\begin{defn}
We define the ring of overconvergent functions
\begin{align*}
\QQ_q \left\langle t, 1/r \right\rangle^{\dag} = 
\Biggl\{\sum_{i,j=0}^{\infty} a_{i,j} \frac{t^i}{r^j} \; : \; 
a_{i,j} \in \QQ_q, \; \exists c > 0 \text{ s.t.}  
\lim_{i+j \rightarrow \infty} \bigl(\ord_p(a_{i,j}) - c(i+j)\bigr) \geq 0
\Biggr\},
\end{align*}
as the $p$-adically meromorphic functions on $\mathbf{P}^1_{\QQ_q}$ that are 
analytic outside of the open disks of radius $\rho$ around the zeros of $r$ 
and the point at infinity for some $\rho<1$. 
\end{defn}

\begin{defn}
We extend the $p$-adic valuation 
to $\QQ_q \langle t, 1/r \rangle^{\dag}$ in the standard way, i.e.\ 
$\ord_p(f)$ is defined as the maximum, over all ways of writing the element $f$ as
$\sum_{i,j=0}^{\infty} a_{i,j} t^i / r^j$,
of the minimum, over $i,j \geq 0$, of $\ord_p(a_{i,j})$. Note that the  
norm on $\QQ_q \langle t, 1/r \rangle^{\dag}$ corresponding to this valuation
is the Gauss norm. We also extend $\ord_p(-)$ to polynomials and matrices over 
$\QQ_q \langle t, 1/r \rangle^{\dag}$, by taking the 
minimum over the coefficients and entries, respectively.
\end{defn}

\begin{assump} \label{assump:S}
From now on we will always assume that $0 \in \mathcal{S}$. 
Note that if this is not the case, then it can be achieved by 
applying a translation.
\end{assump}

\begin{thm} \label{thm:eqphi} 
The matrix $\Phi$ is an element of 
$M_{b \times b}(\QQ_q \langle t, 1/r \rangle^{\dag})$ 
and satisfies the differential equation
\begin{align*}
\left(\frac{d}{dt} + M\right) \Phi &= p t^{p-1} \Phi \sigma(M), &\Phi(0)& = \Phi_0,
\end{align*}
where $\Phi_0$ is the 
matrix of $p^{-1}\Frob_p$ on $\Hrig^n(U_0)$ with respect to the 
basis $[e_0,\dotsc,e_b]$.
\end{thm}

\begin{proof}
That the differential equation is satisfied is an immediate consequence of the 
fact that $\Frob_p$ is a horizontal map of vector bundles with connection, and 
that $\Phi(0)=\Phi_0$ is also clear from Theorem~\ref{thm:frobstruc}. Note that
by a residue disk on $\mathcal{S}$ we mean all points on $\mathcal{S}(\bar{\QQ}_q)$ 
that reduce modulo $p$ to a given point of $S(\bar{\FF}_q)$.
If $M$ does not have any poles in a given residue disk, then $\Phi$ cannot have 
any poles in that residue disk either, by Theorem~\ref{thm:KedlayaTuitman} 
below. Hence the entries of $\Phi$ are contained in 
$\QQ_q \langle t, 1/r \rangle^{\dag}$.
\end{proof}

The deformation method can now be sketched as follows:
\begin{enumerate}[\it Step 1.]
\item Compute the matrix $M$ of the Gauss--Manin connection $\nabla$.
\item Compute the matrix $\Phi_0$ of the action of $p^{-1}\Frob_p$ on 
      $\Hrig^n(U_0)$. If the family is chosen such that $X_0$ is a diagonal 
      hypersurface over a prime field, this can be done as
      explained in Chapter~\ref{sec:Diagonal}.
\item Solve the differential equation from Theorem~\ref{thm:eqphi} for $\Phi$.
\item Substitute the Teichm\"uller lift $\hat{\tau}$ of an element 
      $\tau \in S(\FF_{\mathfrak{q}})$ into $\Phi$ to obtain the 
      matrix $\Phi_{\tau}$ of the action of $p^{-1}\Frob_p$ on 
      $\Hrig^n(U_{\tau})$. Compute the matrix $\Phi_{\tau}^{(a)}$ 
      of the action of $\mathfrak{q}^{-1} \Frob_{\mathfrak{q}}$ 
      on~$\Hrig^n(U_{\tau})$, which is also equal to $(p^{-1}\Frob_p)^a$.
      Use Theorem~\ref{thm:hypersurface} to compute the zeta function 
      $Z(X_{\tau},T)$ of the fibre $X_{\tau}$.
\end{enumerate}

Note that we can only carry out these computations to finite $p$-adic 
precision. Therefore, we need to recall some bounds on the loss of $p$-adic 
precision when multiplying $p$-adic numbers and matrices. 

\begin{prop} \label{prop:productval}
Let $v_1,\dotsc,v_{\ell} \in \ZZ$ and $x_1, \dotsc, x_{\ell} \in \mathbf{Q}_q$,  
$\ell \geq 2$, be such that $\ord_p(x_i) \geq v_i$ for all $i$. Suppose that 
$N \in \ZZ$ satisfies $N \geq \sum_{j=1}^{\ell} v_j$. Let 
$\tilde{x}_1, \dotsc, \tilde{x}_{\ell}$ denote $p$-adic approximations to 
$x_1, \dotsc, x_{\ell}$ such that 
\[
\ord_p\left(x_i - \tilde{x}_i\right) \geq N - \sum_{j \neq i} v_j
\] 
for all $i$.  Then 
\begin{equation*}
\ord_p(x_1 \dotsm x_{\ell} - \tilde{x}_1 \dotsm \tilde{x}_{\ell}) \geq N.
\end{equation*}
\end{prop}

\begin{proof}
For all $i$,
\begin{align*}
\ord_p\left(\tilde{x}_i \right) 
    & \geq \min \{ \ord_p\left(x_i-\tilde{x}_i\right), \ord_p(x_i) \} \\
    & \geq \min \biggl\{ N- \sum_{j \neq i} v_j, \ord_p(x_i)\biggr\} \geq v_i.
\end{align*}
Therefore, we also have that
\begin{equation*}
\ord_p \bigl( (x_{i}-\tilde{x}_{i})
    (\tilde{x}_1 \dotsm \tilde{x}_{i-1} x_{i+1} \dotsm x_{\ell}) \bigr) \geq N,
\end{equation*}
for all $i$. The result now follows by adding these $\ell$ inequalities.
\end{proof}

\begin{prop} \label{prop:matrixproductval}
Let $v_1,\dotsc,v_{\ell} \in \ZZ$ and 
$A_1, \dotsc, A_{\ell} \in M_{b \times b}(\QQ_q)$, $\ell \geq 2$, be 
such that $\ord_p(A_i) \geq v_i$ for all $i$. Suppose that $N \in \ZZ$ 
satisfies $N \geq \sum_{j=1}^{\ell} v_j$. 
Let $\tilde{A}_1, \dotsc, \tilde{A}_{\ell}$ denote $p$-adic approximations 
to $A_1, \dotsc A_{\ell}$ such that
\[
\ord_p\bigl(A_i - \tilde{A}_i\bigr) \geq N - \sum_{j \neq i} v_j
\]
for all $i$.  Then 
\begin{equation}
\ord_p\bigl(A_1 \dotsm A_{\ell} - \tilde{A}_1 \dotsm \tilde{A}_{\ell}\bigr) \geq N.
\end{equation}
\end{prop}

\begin{proof}
We can follow the proof of Proposition~\ref{prop:productval}, 
observing that for matrices $A,B \in M_{b \times b}(\QQ_q)$, we still 
have that $\ord_p(A + B) \geq \min \{\ord_p(A), \ord_p(B)\}$ and 
$\ord_p(AB) \geq \ord_p(A)+\ord_p(B)$.
\end{proof}

\section{Computing the connection matrix}
\label{sec:Connection}

In this section we compute the action of the Gauss--Manin connection $\nabla$ 
on the algebraic de~Rham cohomology $\HdR^{n}(\mathfrak{U}/\mathfrak{S})$ 
of the generic fiber 
$\mathfrak{U}/\mathfrak{S}=\mathcal{U}/\mathcal{S} \otimes \QQ_q$ of the 
complement $\mathcal{U}/\mathcal{S}$ of a family of smooth hypersurfaces 
$\mathcal{X}/\mathcal{S}$ contained in $\mathbf{P}^n_{\mathcal{S}}$ over
some Zariski open subset $\mathcal{S} \subset \mathbf{P}^1_{\ZZ_q}$. Let 
$\mathcal{X}/\mathcal{S}$ be defined by a homogeneous polynomial 
$P \in \ZZ_q[t][x_0,\dotsc,x_n]$ of degree $d$. First we recall how to 
compute in $\HdR^{n}(\mathfrak{U}/\mathfrak{S})$ following the method 
of Griffiths and Dwork.  

\begin{prop} \label{prop:Omega}
Let $\Omega$ denote the $n$-form on $\mathfrak{U}/\mathfrak{S}$ defined by 
\begin{align*}
\Omega = \sum_{i=0}^n (-1)^i x_i d x_0 \wedge \dotsb \wedge \widehat{d x_i} \wedge \dotsb \wedge d x_n.
\end{align*}
The algebraic de~Rham cohomology space $\HdR^{n}(\mathfrak{U}/\mathfrak{S})$ 
is isomorphic to the quotient of the space of closed $n$-forms 
$Q \Omega / P^k$ with $k \in \NN$ and 
$Q \in H^0(\mathfrak{S},\mathcal{O}_{\mathfrak{S}})[x_0, x_1, \dotsc, x_n]$ 
homogeneous of degree $k d - (n + 1)$, by the subspace of exact $n$-forms 
generated by
\begin{equation*} 
\frac{(\partial_i Q) \Omega}{P^k} - k \frac{Q (\partial_i P) \Omega}{P^{k+1}},
\end{equation*}
for all $0 \leq i \leq n$ with $k \in \NN$ and 
$Q \in H^0(\mathfrak{S}, \mathcal{O}_{\mathfrak{S}})[x_0, x_1, \dotsc, x_n]$ 
homogeneous of degree $kd-n$, where $\partial_i$ denotes the partial 
derivative operator with respect to~$x_i$.
\end{prop}

\begin{proof}
The proof is straightforward, for details see \citep[\S 4]{Griffiths1969}.
\end{proof}

The cohomology space $\HdR^{n}(\mathfrak{U}/\mathfrak{S})$ is 
equipped with an increasing filtration by the pole order, for which 
$\mbox{Fil}^k \HdR^{n}(\mathfrak{U}/\mathfrak{S})$ consists of all elements 
that can be represented by $n$-forms $Q \Omega / P^k$ with 
$Q \in H^0(\mathfrak{S},\mathcal{O}_{\mathfrak{S}})[x_0, x_1, \dotsc, x_n]$ 
homogeneous of degree $kd - (n + 1)$.  It follows from a theorem of 
Macaulay~\citep[(4.11)]{Griffiths1969} that 
$\mbox{Fil}^n \HdR^{n}(\mathfrak{U}/\mathfrak{S}) = \HdR^{n}(\mathfrak{U}/\mathfrak{S})$. 
Actually, the reverse filtration $H_i=\mbox{Fil}^{n-i} \HdR^{n}(\mathfrak{U}/\mathfrak{S})$
corresponds to the restriction of the Hodge filtration on 
$\HdR^{n-1}(\mathfrak{X}/\mathfrak{S})$ to $\HdR^{n}(\mathfrak{U}/\mathfrak{S})$ 
by~\citep[(8.6.)]{Griffiths1969}.

As we prefer to perform linear 
algebra operations over a field, we will actually work with the de~Rham 
cohomology vector space $\HdR^{n}\bigl(\mathfrak{U}_{\QQ_q(t)}\bigr)$ of 
the generic fibre 
\[
\mathfrak{U}_{\QQ_q(t)} = \mathfrak{U}/\mathfrak{S} \times_{\mathfrak{S}} \Spec \QQ_q(t).
\] 
We now define an explicit basis of a simple form for 
$\HdR^{n}\bigl(\mathfrak{U}_{\QQ_q(t)}\bigr)$ 
for the families that we are interested in.

\begin{defn} \label{defn:MonBasis}
For $k \in \NN$, we define the following sets of monomials: 
\begin{align*}
F_k & = \{ x^u : u \in \mathbf{N}_{0}^{n+1}, \abs{u} = k d - (n+1) \}, \\
B_k & = \{ x^u : u \in \mathbf{N}_{0}^{n+1}, \abs{u} = k d - (n+1) \text{ and $u_i < d-1$ for all $i$}\}, \\
R_k & = F_k - B_k,
\end{align*}
where $x^u = x_0^{u_0} \dotsm x_n^{u_n}$ and $\abs{u}=\sum_{i=0}^n u_i$. 
We also define 
\begin{equation*}
\cB_k = \{Q \Omega / P^k : Q \in B_k\}, 
\end{equation*}
and write $B = B_1 \cup \dotsb \cup B_n$ and $\cB = \cB_1 \cup \dotsb \cup \cB_n$.
\end{defn}

We will show below that if the family $\mathcal{X}/{\mathcal{S}}$ contains 
a diagonal fibre, then the set $\cB$ forms a basis for 
$\HdR^n\bigl(\mathfrak{U}_{\QQ_q(t)}\bigr)$.

\begin{defn} \label{defn:IndexSets}
For $k \in \NN$, let $C_k^{(0)}$ be the set of monomials of total 
degree $(k-1)d - n$ and then inductively, for $1 \leq j \leq n$, define 
$C_k^{(j)}$ to be the set of monomials in $C_k^{(j-1)}$ except for those 
divisible by $x_{j-1}^{d-1}$.  Moreover, we define the multi-set $C_k$ as 
the disjoint union of $C_k^{(0)}, \dotsc, C_k^{(n)}$.  We shall write an 
element of this multi-set as $(j, g)$, when referring to a monomial~$g$ 
in~$C_k^{(j)}$.
\end{defn}

\begin{lem} \label{lem:bijection}
For all $k \in \NN$, the multi-sets $R_k$ and $C_k$ 
have the same cardinality.
\end{lem}

\begin{proof}
We construct a bijection $R_k \to C_k$, representing the 
monomials by their exponent tuples.  Let $u = (u_0, \dotsc, u_n)$ be an
element of $R_k$.  If $u_0 \geq d-1$, we define the image as
$(u_0-d-1, u_1, \dotsc, u_n) \in C_k^{(0)}$.  More generally, if 
$u_0 < d-1, \dotsc, u_{j-1} < d-1$ and $u_j \geq d-1$, we define the image as 
$(u_0, \dotsc, u_{j-1}, u_j-(d-1), u_{j+1}, \dotsc, u_n) \in C_k^{(j)}$.  
It is easy to verify that this map is indeed a bijection.
\end{proof}

\begin{defn} \label{defn:Deltak}
We define a square matrix $\Delta_k$ with 
row and column index sets $R_k$ and $C_k$ as follows.  
Given $f \in R_k$ and $(j,g) \in C_k$, we set the corresponding entry in 
$\Delta_k$ to be the coefficient of the monomial $f/g$ in $\partial_j P$ if 
$g$ divides $f$ and $0$ otherwise.
\end{defn}

\begin{thm} \label{thm:Isomorphism}
Suppose that the family $\mathcal{X}/\mathcal{S}$ of smooth projective
hypersurfaces given by the polynomial~$P$ in $\ZZ_q[t][x_0, \dotsc, x_n]$ contains 
a diagonal fibre.  For $k \in \NN$ and $0 \leq j \leq n$, let $U_k^{(j)}$ be 
the $\QQ_q(t)$-vector space of polynomials with basis $C_k^{(j)}$, and let $U_k$ 
denote the cartesian product $U_k = U_k^{(0)} \times \dotsb \times U_k^{(n)}$. 
Moreover, let $V_k$ and $W_k$ be the $\QQ_q(t)$-vector spaces of polynomials with 
bases $F_k$ and $R_k$, respectively, and let $\pi \colon V_k \rightarrow W_k$ 
denote the linear map that sends the elements of $B_k$ to zero and the 
elements of $R_k$ to themselves. 
Then the map 
\begin{align}
\phi_k \colon U_k &\to W_k,
&(Q_0, \dotsc, Q_n) &\mapsto \pi \bigl( Q_0 \partial_0 P + \dotsb + Q_n \partial_n P \bigr)
\end{align}
is an isomorphism of $\QQ_q(t)$-vector spaces.
\end{thm}

\begin{proof}
Recall that $R_k$ and $C_k$ have the same cardinality by Lemma~\ref{lem:bijection}.
If $R_k$ and $C_k$ are empty, then $U_k$ and $W_k$ are the zero vector spaces, 
and the theorem holds trivially. So suppose that $R_k$ and $C_k$ are nonempty. 
It is immediate that $\Delta_k$ is the matrix 
representing $\phi_k$ with respect to the bases $C_k$ and $R_k$ of $U_k$ and $W_k$, 
respectively.

The assumption that the family~$\mathcal{X}/\mathcal{S}$ contains a diagonal 
hypersurface means that for some~$t_0 \in \mathcal{S}(\ZZ_q)$, 
the fibre $\mathcal{X}_{t_0}$ is defined by a polynomial of the form 
\begin{equation*}
P_{t_0}(x_0, \dotsc, x_n) = a_0 x_0^d + \dotsb + a_n x_n^d
\end{equation*}
with $a_0, \dotsc, a_n \in \ZZ_q^{\times}$.

We now show that the determinant of $\Delta_k$ is nonzero.  Since 
evaluation of the matrix at \mbox{$t = t_0$} commutes with computing the 
determinant, it suffices to show that the determinant of 
$(\Delta_k) \big |_{t=t_0}$ is nonzero. Since, for $0 \leq j \leq n$, 
we have $\partial_j P_{t_0} (x_0, \dotsc, x_n) = d a_j x_j^{d-1}$, there is 
precisely one nonzero entry in each column and each row of~$\Delta_k$.  
Namely, in column $(j, g) \in C_k$ and row $g x_j^{d-1} \in R_k$ there is 
the nonzero entry $d a_j$. Note that this also implies that 
$(\Delta_k) \big |_{t=t_0} \in \ZZ_{q}^{\times}$.
\end{proof}

We can use Theorem~\ref{thm:Isomorphism} to give a routine {\sc Decompose}, 
formalised in Algorithm~\ref{alg:Decompose}, which given 
$Q \in \QQ_q(t)[x_0, \dotsc, x_n]$ homogeneous of degree $kd - (n+1)$
returns an expression 
\begin{equation*}
Q = Q_0 \partial_0 P + \dotsb + Q_n \partial_n P + \gamma_k
\end{equation*} 
with $Q_0, \dotsc, Q_n \in \QQ_q(t)[x_0, \dotsc, x_n]$ homogeneous of 
degree $kd-n$ and $\gamma_k$ in the $\QQ_q(t)$-span of~$B_k$. We can in turn 
use {\sc Decompose} to furnish another routine {\sc Reduce}, formalised 
in Algorithm~\ref{alg:PoleRed}, which given a closed $n$-form $Q\Omega/P^k$ 
with $Q \in \QQ_q(t)[x_0, \dotsc, x_n]$ homogeneous of degree $kd - (n+1)$ returns 
an expression
\begin{equation*}
\frac{Q \Omega}{P^k} \equiv \frac{\gamma_{1} \Omega}{P^{1}} 
                            + \dotsb + \frac{\gamma_n \Omega}{P^n},
\end{equation*}
with $\gamma_i$ in the $\QQ_q(t)$-span of $B_i$ for $1 \leq i \leq n$ and 
where $\equiv$ denotes equality in cohomology.

\begin{algorithm}
\caption{Obtain coordinates in the Jacobian ideal modulo basis elements}
\label{alg:Decompose}
\begin{algorithmic}
\Require $P$ in $\ZZ_q[t][x_0, \dotsc, x_n]$ homogeneous of degree~$d$, 
         defining a family $\mathcal{X}/\mathcal{S}$ of smooth projective 
         hypersurfaces that contains a diagonal fibre, 
         $Q \in \QQ_q(t)[x_0, \dotsc, x_n]$ homogeneous of degree $kd - (n+1)$.
\Ensure  $Q_0, \dotsc, Q_n \in \QQ_q(t)[x_0, \dotsc, x_n]$ homogeneous of degree 
         $k(d-1)-n$, and $\gamma_k$ in the $\QQ_q(t)$-span of $B_k$, such that 
         $Q = Q_0 \partial P_0 + \dotsb + Q_n \partial_n P +\gamma_k$.
\Procedure{Decompose}{$P,Q$}
\State \begin{compactenum}[{\hspace{1em}} 1.] \vspace{-1.24em}
\item Let $w$ be the vector of length $\abs{R_k}$ such that the entry 
      corresponding to $x^u \in R_k$ is the coefficient of 
      $x^u$ in $Q$.
\item Solve for the unique vector $v$ of length $\abs{C_k}$ satisfying 
      $\Delta_k v = w$.  We write $v$ 
      as $\bigl(v^{(0)}, \dotsc, v^{(n)}\bigr)$, where $v^{(j)}$ is 
      a vector of length $\abs{C_k^{(j)}}$ for $0 \leq j \leq n$,
      and let $v_g^{(j)}$ be the entry in $v^{(j)}$ corresponding 
      to $g \in C_k^{(j)}$.
\item For $0 \leq j \leq n$, compute $Q_j \gets \sum_{g \in C_k^{(j)}} v_g^{(j)} g$.
\item Set $\gamma_k \gets Q-(Q_0 \partial P_0 + \dotsb + Q_n \partial_n P)$.
\item \textbf{return} $Q_0, \dotsc, Q_n,\gamma_k$      
\EndProcedure
\end{compactenum}
\end{algorithmic}
\end{algorithm}

\begin{algorithm}
\caption{Reduce $Q \Omega / P^k$ in $\HdR^n\bigl(\mathfrak{U}_{\QQ_q(t)}\bigr)$}
\label{alg:PoleRed}
\begin{algorithmic}
\vspace{1mm}
\Require $P$ in $\ZZ_q[t][x_0, \dotsc, x_n]$ homogeneous of degree~$d$, 
         defining a family $\mathcal{X}/\mathcal{S}$ of smooth projective 
         hypersurfaces that contains a diagonal fibre, $Q \in \QQ_q(t)[x_0, \dotsc, x_n]$ 
         homogeneous of degree $kd - (n+1)$.
\Ensure  $\gamma_i$ in the $\QQ_q(t)$-span of $B_i$ for $1 \leq i \leq n$, with  
         $Q \Omega / P^k \equiv \gamma_{1} \Omega / P^{1} + \dotsb + \gamma_n \Omega / P^n$.
\Procedure{Reduce}{$P,Q$}
\While{$k \geq n+1$}
\State $Q_0, \dotsc, Q_n, 0 \gets \Call{Decompose}{Q}$
\State $k \gets k-1$
\State $Q \gets k^{-1} \sum_{i=0}^n \partial_i Q_i$
\EndWhile
\While{$Q \not \in \QQ_q(t)$-span of $B_k$}
\State $Q_0, \dotsc, Q_n, \gamma_k \gets \Call{Decompose}{Q}$
\State $k \gets k-1$
\State $Q \gets k^{-1} \sum_{i=0}^n \partial_i Q_i$
\EndWhile
\If{$Q \neq 0$}
\State $\gamma_{k} \gets Q$
\State $k \gets k-1$
\EndIf
\State $\gamma_{1}, \dotsc, \gamma_{k} \gets 0$
\State \textbf{return} $\gamma_{1}, \dotsc, \gamma_n$
\EndProcedure
\end{algorithmic}
\end{algorithm}

We now establish that the set~$\cB$ indeed forms a basis for 
$\HdR^n\bigl(\mathfrak{U}_{\QQ_q(t)}\bigr)$, as announced before.  
We start with an auxiliary result describing the cardinality of 
the set~$\cB$.

\begin{prop} \label{prop:BasisSize}
The set $\cB$ has cardinality
\begin{equation*}
\frac{1}{d} \left((d-1)^{n+1} + (-1)^{n+1}(d-1) \right).
\end{equation*}
\end{prop}

\begin{proof}
First note that if we denote
\begin{align*}
V   &= \Bigl\{(u_0,\dotsc,u_n) \in (\ZZ/d\ZZ)^{n+1} : \sum_{j=0}^n u_j = -(n+1) \Bigr\}, \\
W_j &= \Bigl\{(u_0,\dotsc,u_n) \in (\ZZ/d\ZZ)^{n+1} : u_j = -1 \Bigr\},
\end{align*}
then  $\cB$ is in one-to-one correspondence with the set $V-(W_0 \cup \dotsb \cup W_n)$. 
Now by the inclusion-exclusion principle, 
\begin{align*}
\abs{V \cap (W_0 \cup \dotsb \cup W_n)} 
& = \sum_{j=0}^n \abs{V \cap W_j} 
    - \sum_{0 \leq j < k \leq n} \abs{V \cap W_j \cap W_k} \\
& \quad + \dotsb + (-1)^{n} \abs{V \cap W_0 \cap \dotsb \cap W_n} \\
& = {n+1 \choose 1} d^{n-1} -{n+1 \choose 2} d^{n-2} 
    + \dotsb + (-1)^{n-1} {n+1 \choose n} + (-1)^{n} \\
& = \frac{1}{d} \left(d^{n+1}+(-1)^{n+1} - (d-1)^{n+1}\right)+(-1)^n,
\end{align*}
so that
\begin{align*}
\abs{V-(W_0 \cup \dotsb \cup W_n)}&=\abs{V}-\abs{V \cap (W_0 \cup \dotsb \cup W_n)} \\
&= d^n - \frac{1}{d} \bigl(d^{n+1}+(-1)^{n+1} - (d-1)^{n+1}+d (-1)^n \bigr) \\
&= \frac{1}{d} \left((d-1)^{n+1} + (-1)^{n+1}(d-1) \right),
\end{align*}
and the proof is complete.
\end{proof}

\begin{prop} \label{prop:rankcoho}
The rank of $\HdR^n(\mathfrak{U}/\mathfrak{S})$ is
\[
\frac{1}{d} \left((d-1)^{n+1} + (-1)^{n+1}(d-1) \right).
\]
\end{prop}

\begin{proof}
Let $\mathfrak{U}_m/\mathfrak{S}$ denote the complement of a family
$\mathfrak{X}_m/\mathfrak{S}$ of smooth projective hypersurfaces of 
degree~$d$ in $\mathbf{P}^m_{\mathfrak{S}}$ and let $b(d,m)$ denote 
the rank of $\HdR^m(\mathfrak{U}_m/\mathfrak{S})$. It is known that 
$b(d,m)$ only depends on $d$, $m$. Moreover, taking $x=y=z$ 
in~\citep[Corollaire~2.4~(i)]{sga7}, we find that
\[
\sum_{m=1}^{\infty} b(d,m) x^{m-1} 
  = \frac{(d-1)}{(1+x)(1-(d-1)x)} 
  = \frac{1}{xd} \left( \frac{d-1}{1-(d-1)x} - \frac{d-1}{1+x} \right)
\]
as formal power series. From this the result follows easily. 
\end{proof}

\begin{thm} \label{thm:Basis}
Suppose that the family of smooth projective hypersurfaces 
$\mathcal{X}/\mathcal{S}$ contains a diagonal fibre.  Then the 
set~$\cB$ from Definition~\ref{defn:MonBasis} is a basis for 
the $\QQ_q(t)$-vector space $\HdR^n\bigl(\mathfrak{U}_{\QQ_q(t)}\bigr)$.
\end{thm}

\begin{proof}
We already know that $\HdR^n\bigl(\mathfrak{U}_{\QQ_q(t)}\bigr)$ is 
spanned by the classes of the $n$-forms $Q \Omega / P^k$ with 
$Q \in \QQ_q(t)[x_0, \dotsc, x_n]$ homogeneous of degree 
$kd - (n+1)$ for $k \in \NN$. Applying Algorithm~\ref{alg:PoleRed}, we 
obtain an expression for the class of $Q \Omega / P^k$ as a 
$\QQ_q(t)$-linear combination of elements in $\cB$.  This shows that 
$\cB$ spans the vector space $\HdR^n\bigl(\mathfrak{U}_{\QQ_q(t)}\bigr)$. 
However, by the two propositions above, the dimension of 
$\HdR^n\bigl(\mathfrak{U}_{\QQ_q(t)}\bigr)$ is equal to the cardinality of
$\cB$, so that $\cB$ is linearly independent as well.
\end{proof}

\begin{rem} \label{rem:hnumbers}
Let $H_i$ denote the restriction of the Hodge filtration on 
$\HdR^{n-1}(\mathfrak{X}/\mathfrak{S})$ to $\HdR^{n}(\mathfrak{U}/\mathfrak{S})$
and write 
\[
h^{i,n-1-i}=\mbox{rank}(H_i/H_{i+1})
\] 
for the corresponding Hodge numbers.
Recall that $H_i=\mbox{Fil}^{n-i} \HdR^{n}(\mathfrak{U}/\mathfrak{S})$
for the filtration by the pole order defined above. Applying 
the loop in Algorithm~\ref{alg:PoleRed} just once, to lower the pole 
order from $k$ to $k-1$, we see that $\cB_k$ spans 
\[
\mbox{Fil}^{k} \HdR^{n}\bigl(\mathfrak{U}_{\QQ_q(t)}\bigr) / \mbox{Fil}^{k-1} \HdR^{n}\bigl(\mathfrak{U}_{\QQ_q(t)}\bigr)
\]
for all $1 \leq k \leq n$. Hence it follows that $\card{B_k} \geq h^{n-k,k-1}$ , but since
\[
\sum_{k=1}^n \card{B_k} = \dim \HdR^n\bigl(\mathfrak{U}_{\QQ_q(t)}\bigr) = \sum_{k=1}^n h^{n-k,k-1},
\] 
this implies that $\card{B_k} = h^{n-k,k-1}$ for all $1 \leq k \leq n$. So the Hodge numbers
can be read off from the basis $\cB$.
\end{rem}

We now describe the action of the Gauss--Manin connection~$\nabla$ on 
$\HdR^n\bigl(\mathfrak{U}_{\QQ_q(t)}\bigr)$.  Suppose that we are given a basis element 
$x^u \Omega / P^k \in \cB_k$.  Following the description in 
Section~\ref{sec:Background}, we compute
\begin{equation} \label{eqn:nabla}
\nabla \left(\frac{x^u \Omega}{P^k}\right) \equiv 
    dt \otimes \frac{- k x^u (\partial P / \partial t) \Omega}{P^{k+1}},
\end{equation}
where $\equiv$ denotes equality in 
$\Omega_{\QQ_q(t)} \otimes \HdR^n\bigl(\mathfrak{U}_{\QQ_q(t)}\bigr)$. 
We apply Algorithm~\ref{alg:PoleRed} in order to write
\begin{equation}
dt \otimes \frac{- k x^u (\partial P / \partial t) \Omega}{P^{k+1}} \equiv 
dt \otimes \left( \frac{\gamma_{1}}{P} + \dotsb + \frac{\gamma_n}{P^n} \right) \Omega,
\end{equation}
where $\gamma_i$ is an element in the $\QQ_q(t)$-span of~$B_i$ for $1 \leq i \leq n$. 

\begin{rem} \label{rem:precgm}
For the matrix $M$ of $\nabla$ to be correct to $p$-adic precision $N_M$, we have to 
carry out this computation to a somewhat higher working precision $N'_M$ due to the 
precision loss in Algorithm~\ref{alg:PoleRed}. 
For all $k \in \NN$, by definition $\ord_p(\Delta_k) \geq 0$, and from the proof
of Theorem~\ref{thm:Isomorphism} we know that $\ord_p(\det(\Delta_k))=0$. 
Therefore, there is no precision loss 
in Algorithm~\ref{alg:Decompose}, and the only precision loss in Algorithm~\ref{alg:PoleRed} 
comes from dividing by $k-1$ for $k=2,\dotsc,n+1$. Hence it is sufficient to take 
\begin{equation*}
N_M'=N_M + \ord_p(n!).
\end{equation*}
\end{rem}
This is formalised in Algorithm~\ref{alg:Connection}.

\begin{algorithm}
\caption{Compute the Gauss--Manin connection matrix}
\label{alg:Connection}
\begin{algorithmic}
\Require $P$ in $\ZZ_q[t][x_0, \dotsc, x_n]$ homogeneous of degree~$d$, 
         defining a family $\mathfrak{X}/\mathfrak{S}$ of smooth projective 
         hypersurfaces that contains a diagonal fibre, $p$-adic precision $N_M$.
\Ensure  The matrix~$M$ of $\nabla$ with respect to $\cB$ to $p$-adic precision $N_M$.
\Procedure{GMConnection}{$P,N_M$}
\State $N'_M \gets N_M + \ord_p(n!)$
\State \textit{Execute the following steps with $p$-adic working precision $N'_M$}:
\For{$g \in B$} 
\State $k \gets  \bigl(\deg(g)+(n+1)\bigr)/d$
\State $Q \gets  - k g (\partial P / \partial t)$
\State $\gamma_{1}, \dotsc, \gamma_n \gets$ {\sc Reduce($P,Q$)} 
\For{$f \in B$}
\State $l \gets \bigl(\deg(f)+(n+1)\bigr)/d$
\State $M_{f,g} \gets$ coefficient of $f$ in $\gamma_l$
\EndFor
\EndFor
\Return $M$
\EndProcedure
\end{algorithmic}
\end{algorithm}

\begin{defn} \label{defn:resultant}
We define the polynomial $R \in \ZZ_q[t]$ by
\[
R = \prod_{k=2}^{n+1}  \det(\Delta_k).
\]
\end{defn}

\begin{prop} \label{thm:denom}
The matrix $M$ of $\nabla$ with respect to $\cB$ is of the form
$H/R$, with $H \in M_{b \times b}(\QQ_q[t])$.
\end{prop}

\begin{proof}
The only time in Algorithm~\ref{alg:Connection} that nonconstant denominators 
are introduced is when the subroutine {\sc{Reduce}} calls its subroutine 
{\sc{Decompose}} and a linear system $\Delta_k v = w$ is solved. Since this 
happens only for $k=2, \dotsc, n+1$ and at most once for each 
such~$k$, the result is clear.
\end{proof}

\begin{rem}
Recall that $\cB$ is a basis for $\HdR^n\bigl(\mathfrak{U}_{\QQ_q(t)}\bigr)$ 
but not necessarily for $\HdR^n(\mathfrak{U}/\mathfrak{S})$. However, if the 
zero locus of $R$ in $\mathbf{P}^1_{\ZZ_q}$ does not intersect $\mathcal{S}$, 
then $\cB$ is a basis for $\HdR^n(\mathfrak{U}/\mathfrak{S})$. That it 
spans the cohomology can be seen by applying Algorithm~\ref{alg:PoleRed}, and that it 
is linearly independent follows because it is so over $\QQ_q(t)$. Therefore, it will 
be convenient to choose $\mathcal{S}$ smaller, so that this condition is satisfied. 
\end{rem}

\begin{assump} \label{assump:R}
From now on we assume that the zero locus of~$R$ in~$\mathbf{P}^1_{\ZZ_q}$ 
does not intersect $\mathcal{S}$. In particular, this implies that $\cB$ is 
a basis for $\HdR^n(\mathfrak{U}/\mathfrak{S})$.
\end{assump}

\begin{rem}
Note that in Algorithm~\ref{alg:Connection}, we solve $\BigOh(\card{\cB})$~linear 
systems given by the matrices $\Delta_k$ for $k = 2, \dotsc, n+1$.  In practice, 
it is important to take advantage of this, e.g.\ by computing the decomposition 
$\Delta_k = LUP$, where $L$ is lower triangular, $U$ is upper triangular, and 
$P$ is a permutation matrix, which is guaranteed to exist for any square matrix. 
Since, moreover, the matrices~$\Delta_k$ contain many zeros, methods from sparse 
linear algebra can be used to compute such a decomposition.  Then, every call to 
{\sc{Decompose}} reduces to solving a lower and an upper diagonal linear system.
\end{rem}

In the rest of this section we will freely use the definitions and notation 
from Section~\ref{sec:Background}. We let all of our matrices be defined with 
respect to the basis $\cB$. We have seen that there exists a Frobenius 
structure $\Frob_p$ on $\HdR^n(\mathfrak{U}/\mathfrak{S})$ and let 
$\Phi \in M_{b \times b}(\QQ_q \langle t, 1/r \rangle^{\dag})$ 
denote the matrix of the action of $p^{-1}\Frob_p$, where $r$ is the 
denominator of the connection matrix $M$. Recall that 
$\FF_{\mathfrak{q}}/\FF_q$ 
denotes a finite field extension. 

We need a priori lower bounds on the $p$-adic valuations of the matrices 
$\Phi$ and $\Phi^{-1}$ to bound the loss of $p$-adic precision in our 
computations. We will now recall how such bounds can be obtained 
following~\citep{AbbottKedlayaRoe2006}.

\begin{thm} \label{thm:deltabound}
For any finite field extension $\FF_{\mathfrak{q}}/\FF_q$ and
all $\tau \in S(\FF_{\mathfrak{q}})$, 
there exists a matrix $W_{\tau} \in M_{b \times b}(\QQ_{\mathfrak{q}})$ such that 
\begin{align*}
\ord_p(W_{\tau}) &\geq -\sum_{i=1}^{n-1} \floor{\log_p(i)},
&\ord_p\bigl(W_{\tau} \Phi_{\tau} \sigma(W_{\tau}^{-1})\bigr) &\geq  0, \\
\ord_p(W_{\tau}^{-1}) &\geq -\ord_p \bigl((n-1)! \bigr),
&\ord_p\bigl(W_{\tau} \Phi_{\tau}^{-1} \sigma(W_{\tau}^{-1})\bigr) &\geq  -(n-1).
\end{align*}
\end{thm}

\begin{proof}
Let $\Lambda_{\tau, crys}$ be the image of the integral logarithmic de Rham 
cohomology space 
$\HdR^n\bigl(\mathbf{P}_{\mathbf{Z}_{\mathfrak{q}}}^n,\mathcal{X}_{\hat{\tau}}\bigr)$ 
in the rigid cohomology space $\Hrig^n(U_{\tau})$. It is known that
$\Lambda_{\tau, crys} \otimes \QQ_{\mathfrak{q}} \cong \Hrig^n(U_{\tau})$, 
so that $\Lambda_{\tau, crys}$ is a lattice in $\Hrig^n(U_{\tau})$.  Let 
$\Lambda_{\tau, mon}$ be the $\ZZ_{\mathfrak{q}}$-module in $\Hrig^n(U_{\tau})$ 
generated by the classes of $x^u \Omega/P^n$ with $u \in \NN_0^{n+1}$ such that 
$\sum_{i=0}^n u_i = nd-(n+1)$. Since these classes generate $\Hrig^n(U_{\tau})$, we have 
that $\Lambda_{\tau,mon}$ is a lattice in $\Hrig^n(U_{\tau})$ as well. We know 
that
\[
\Lambda_{\tau,crys} \subset \Lambda_{\tau,mon} \subset p^{-\sum_{i=1}^{n-1} \floor{\log_p(i)}} \Lambda_{\tau,crys}.
\]
The inclusion on the left is \citep[Lemma 3.4.3]{AbbottKedlayaRoe2006}, and the 
one on the right is \citep[Proposition 3.4.6]{AbbottKedlayaRoe2006}.

Let $\Lambda_{\tau,\cB}$ be the $\ZZ_{\mathfrak{q}}$-module in $\Hrig^n(U_{\tau})$ 
generated by the basis $\cB$. Since $\cB$ spans $\Hrig^n(U_{\tau})$, we have that 
$\Lambda_{\tau,\cB}$ is also a lattice in $\Hrig^n(U_{\tau})$. Now we know that
\[
(n-1)! \Lambda_{\tau,mon} \subset \Lambda_{\tau,\cB} \subset \Lambda_{\tau,mon}. 
\]
The inclusion on the left follows by explicitly reducing the generators of 
$\Lambda_{\tau,mon}$ to the basis $\cB$ with Algorithm~\ref{alg:PoleRed} using 
that $R(\hat{\tau}) \in \ZZ_{\mathfrak{q}}^{\times}$, and the inclusion on the right is clear.

Combining these inclusions of lattices, we find that
\begin{equation} \label{eq:lattices}
(n-1)! \Lambda_{\tau,crys} \subset \Lambda_{\tau,\cB} \subset p^{-\sum_{i=1}^{n-1} \floor{\log_p(i)}} \Lambda_{\tau,crys}.
\end{equation}
Now let $[d_1, \dotsc, d_b]$ be a $\ZZ_{\mathfrak{q}}$-basis for 
$\Lambda_{\tau,crys}$ and let $W_{\tau} \in M_{b \times b}(\QQ_q)$ be 
the matrix in which the $i$-th column 
consists of the coordinates of $d_i$ with respect to the basis~$\cB$. From the 
inclusions~\eqref{eq:lattices} it is clear that 
\begin{eqnarray*}
\ord_p(W_{\tau}) &\geq& -\sum_{i=1}^{n-1} \floor{\log_p(i)}, \\
\ord_p(W_{\tau}^{-1}) &\geq& -\ord_p((n-1)!).
\end{eqnarray*}
Note that $W_{\tau} \Phi_{\tau} \sigma(W_{\tau}^{-1})$ is the matrix of 
$p^{-1}\Frob_{p}$ on $\Hrig^n(U_{\tau})$ with respect to the basis 
$[d_1,\dotsc,d_b]$. Now $\Lambda_{\tau,crys}$ is contained in the crystalline 
cohomology space $H^{n-1}_{crys}(X_{\tau})$, which maps to itself 
under~$\Frob_p$. So by the short exact sequence~\eqref{eqn:excision}, the 
lattice $\Lambda_{\tau,crys}$ maps to itself under $p^{-1}\Frob_{p}$, and
\[
\ord_p(W_{\tau} \Phi_{\tau} \sigma(W_{\tau}^{-1})) \geq 0.
\]
Similarly, note that $W_{\tau} \Phi_{\tau}^{-1} \sigma(W_{\tau}^{-1})$ is 
the matrix of $p\Frob_p^{-1}$ on $\Hrig^n(U_{\tau})$ with respect to the 
basis $[d_1,\dotsc,d_b]$. By Poincar\'e duality, the map $p^{n-1}\Frob_p^{-1}$ 
maps the crystalline cohomology space $H^{n-1}_{crys}(X_{\tau})$ to itself. 
So by the short exact sequence~\eqref{eqn:excision}, the lattice 
$\Lambda_{\tau,crys}$ maps to itself under $p^{n-1} (p\Frob_p^{-1})$, and 
\begin{equation*}
\ord_p(W_{\tau} \Phi_{\tau}^{-1} \sigma(W_{\tau}^{-1})) \geq -(n-1). \qedhere
\end{equation*}
\end{proof}

\begin{defn} \label{defn:delta}
Define the nonnegative integer
\[
\delta = \ord_p((n-1)!)+\sum_{i=1}^{n-1} \floor{\log_p(i)}.
\]
\end{defn}

\begin{cor} \label{cor:delta} We have 
\begin{align*}
\ord_p(\Phi) &\geq -\delta, \\
\ord_p(\Phi^{-1}) &\geq -\delta-(n-1).
\end{align*}
\end{cor}

\begin{proof}
For every $\tau \in \mathcal{S}(\bar{\FF}_{\mathfrak{q}})$, 
we can apply Theorem~\ref{thm:deltabound} 
to obtain
\begin{align*}
\ord_p(\Phi_{\tau}) &\geq \ord_p(W^{-1}_{\tau}) + 
                          \ord_p(W_{\tau} \Phi_{\tau} \sigma(W_{\tau}^{-1})) + 
                          \ord_p(\sigma(W_{\tau})) \\
                    &\geq -\delta,
\end{align*}
and also
\begin{align*}
\ord_p(\Phi^{-1}_{\tau}) &\geq \ord_p(W^{-1}_{\tau}) + 
                               \ord_p(W_{\tau} \Phi^{-1}_{\tau} \sigma(W_{\tau}^{-1})) + 
                               \ord_p(\sigma(W_{\tau})) \\
                         &\geq -\delta - (n-1).
\end{align*}
As these inequalities hold for infinitely many~$\tau \in S(\bar{\FF}_q)$ 
we obtain the required bounds.
\end{proof}


\section{Frobenius on diagonal hypersurfaces}
\label{sec:Diagonal}

\subsection{A formula for $\Phi_0$}

We consider a single smooth 
projective diagonal hypersurface~$\mathcal{X}_0$ over $\ZZ_p$ defined by 
a polynomial $P_0 \in \ZZ_p[x_0, \dotsc, x_n]$ of the form
\begin{equation*}
P_0(x_0, x_1, \dotsc, x_n) = 
    a_0 x_0^d + a_1 x_1^d + \dotsb + a_n x_n^d,
\end{equation*}
where $a_0, a_1, \dotsc, a_n \in \ZZ_p^{\times}$ and $p \nmid d$. 
Let $\mathfrak{X}_0 = \mathcal{X}_0 \otimes_{\ZZ_p} \QQ_p$ denote the generic 
fibre of $\mathcal{X}_0$ and let $\mathcal{U}_0$ and $\mathfrak{U}_0$ be the 
complements of $\mathcal{X}_0$ and $\mathfrak{X}_0$, respectively. 
We fix our choice of basis~$\cB$ for $\HdR^{n}(\mathfrak{U}_0)$ 
as in Definition~\ref{defn:MonBasis}. 
Our goal is to compute 
the matrix~$\Phi_0$ representing the action of $p^{-1} \Frob_p$ on 
$\Hrig^n(U_0) \cong \HdR^n(\mathfrak{U}_0)$, with respect to the basis~$\cB$, 
to $p$-adic precision~$N_{\Phi_0}$. It will turn out that this matrix has only one
nonzero element in every row and column.

Let $u = (u_0, \dotsc, u_n)$ and $v = (v_0, \dotsc, v_n)$ be tuples 
of integers such that $x^u, x^v \in B$ and $p (u_i+1) \equiv v_i+1 \bmod{d}$,
for all $i$. Furthermore, let $k(u) \in \NN$ denote the positive integer such that 
\[
k(u) d -(n+1) = \sum_{i=0}^n u_i.
\] 
Finally, for all $l \in \QQ$ and integers $r \geq 0$, let the rising factorial $\prod_{j=0}^{r-1} (l + j)$ be
denoted by $(l)_r$.

\begin{defn} \label{defn:alpha}
Let $u, v \in \ZZ^{n+1}$ be such that we have
$x^u, x^v \in B$ and 
$p (u_i + 1) \equiv v_i + 1 \bmod{d}$ for all~$i$. 
We define
\begin{equation*}
\alpha_{u,v} = \prod_{i=0}^n a_i^{(p (u_i + 1) - (v_i + 1))/d} 
    \biggl( \sum_{m,r} \left( \frac{u_i+1}{d} \right)_r 
        \sum_{j=0}^{r} \frac{\bigl(p a_i^{p-1}\bigr)^{r-j}}{(m-pj)!j!} \biggr),
\end{equation*}
where the sum in the $i$-th factor of the product is over all integers $m, r \geq 0$  
that satisfy $p(u_i+1)-(v_i+1)=d(m-pr)$.
\end{defn}

The following lemma will be useful in what follows.

\begin{lem} \label{lem:mpr}
Let $x^u, x^v \in B$, $0 \leq i \leq n$ and $m, r \geq 0$ integers, such that 
\[
p(u_i + 1) - (v_i + 1) = d(m-pr).
\] 
Then we have $r = \floor{m/p}$.
\end{lem}

\begin{proof}
Since $0 \leq u_i, v_i \leq d-2$, we obtain
\[
p-(d-1) \leq p(u_i + 1) - (v_i + 1) \leq p(d-1)-1,
\]
and from this the result follows, also using that $m - pr \in \ZZ$.
\end{proof}

The entries of $\Phi_0$ can be expressed in terms of the $\alpha_{u,v}$ in the following way.

\begin{thm} \label{thm:01-03-diagfrob}
Let $\omega_1$ denote an element of $\cB$ corresponding to a tuple 
$u \in \ZZ^{n+1}$ and let $\omega_2$ denote the unique element of~$\cB$ 
corresponding to a tuple $v \in \ZZ^{n+1}$ such that
$p (u_i + 1) \equiv v_i + 1 \bmod{d}$ for all $i$. Then
\begin{equation} \label{eq:diagfrobformula}
p^{-1} \Frob_p (\omega_1) = 
    (-1)^{k(v)} \frac{(k(v) - 1)!}{(k(u) - 1)!} p^{n-k(u)} \alpha_{u,v}^{-1} \cdot \omega_2.
\end{equation}
\end{thm}

\begin{proof} We start with some definitions.
First, let $\QQ_p(\pi)$ denote the totally ramified extension of $\QQ_p$ with $\pi^{p-1} = -p$ 
and extend the $p$-adic valuation such that \mbox{$\ord_p(\pi) = (p-1)^{-1}$}. For $0 \leq i \leq n$, 
define sets
\begin{align*}
W^{(i)} &= \{ (w_0,\cdots,w_n) \in \NN_0^{n+1} \colon w_j \geq 1 \text{ for all } j \neq i \text{ and } \lvert w \rvert \equiv 0 \bmod{d} \},
\end{align*}
with $\lvert w \rvert=\sum_{i=0}^n w_i$. Note that $\lvert w \rvert=k(w-1)$, 
where $w-1$ denotes the tuple $(w_0-1,\ldots,w_n-1)$.  Furthermore, still 
for $0 \leq i \leq n$, define $\QQ_p(\pi)$-vector spaces
\begin{align*}
\mathcal{L}^{(i)} &= \Bigl\{ \sum_{w \in W^{(i)}} b_w x^{w} \; : \; 
b_{w} \in \QQ_p(\pi), \; \exists c > 0 \text{ s.t.}  
\lim_{\lvert w \rvert \rightarrow \infty} \bigl(\ord_p(b_w) - c \cdot \lvert w \rvert \bigr) \geq 0 \Bigr\}
\end{align*}
and let $\mathcal{L} = \cap_{i=0}^n \mathcal{L}^{(i)}$ be the intersection of the 
$\mathcal{L}^{(i)}$. Moreover, let $D_i \colon \mathcal{L}^{(i)} \rightarrow \mathcal{L}$ 
denote the differential operator $D_i = x_i \partial_i +  \pi x_i \partial_i(P_0)$.
Finally, let $\psi_p, \alpha \colon \mathcal{L} \rightarrow \mathcal{L}$ be the 
$\QQ_p(\pi)$-linear maps defined by:
\begin{align*}
\psi_p(x^w) &= 
\begin{cases}
x^{w/p} &\text{ if } p \mid w_i \text{ for all } 0 \leq i \leq n \\
0                        &\text{ otherwise}
\end{cases} \\ 
\alpha(x^w) &= \psi_p \Bigl( \exp\Bigl(\pi \bigl(P_0(x_0,\ldots,x_n)-P_0(x_0^p,\ldots,x_n^p)\bigl) \Bigr) \cdot x^w \Bigr)
\intertext{and let $\mathcal{R} \colon \mathcal{L} \rightarrow \Hrig^n(U_0) \otimes \QQ_p(\pi)$ 
denote the $\QQ_p(\pi)$-linear map defined by}
\mathcal{R}(x^w) &= \frac{(\lvert w \rvert-1)!}{(-\pi)^{\lvert w \rvert-1}} \frac{x^{w-1} \Omega}{P_0^{\lvert w \rvert}},
\end{align*}
where $\Omega$ is defined as in Proposition~\ref{prop:Omega}.

Then by \citep[Theorem 2.15]{katz}, the diagram
\[
\begin{CD}
\mathcal{L} / \sum_{i=0}^n D_i \mathcal{L}^{(i)} @>  \; \; \; \; \; p^{-n} \alpha \; \; \; \;   >> \mathcal{L} / \sum_{i=0}^n D_i \mathcal{L}^{(i)}  \\
@VV \mathcal{R} V  @V \mathcal{R} VV \\
\Hrig^n(U_0) \otimes \QQ_p(\pi) @> \left(p^{-1} \Frob_p \right)^{-1}  >> \Hrig^n(U_0) \otimes \QQ_p(\pi)
\end{CD} 
\]
commutes and by \citep[Corollary 1.15]{katz} the vertical maps are isomorphisms. 
This implies that~\eqref{eq:diagfrobformula} is equivalent to
\[
\alpha(x^{v+1}) \equiv \pi^{pk(u)-k(v)} \alpha_{u,v} \cdot x^{u+1} \bmod{\sum_{i=0}^n D_i \mathcal{L}^{(i)}},
\]
which we will now prove.

Note that so far we have not used that $P_0$ is diagonal. Since
$P_0(x_0, x_1, \dotsc, x_n)$ is equal to $a_0 x_0^d + a_1 x_1^d + \dotsb + a_n x_n^d$,
we have
\begin{align*}
\exp \Bigl(\pi \bigl(P_0(x_0,\ldots,x_n)- P_0(x_0^p,\ldots,x_n^p)\bigl) \Bigr) 
    = \prod_{i=0}^n \exp \Bigl( \pi a_i (x_i^d - x_i^{dp}) \Bigr)
\end{align*}
and also
\begin{align*}
D_i =x_i \partial_i + (\pi a_i d) x_i^d. 
\end{align*}
For $0 \leq i \leq n$, we denote
\begin{align*}
\QQ_p(\pi) \left\langle x_i \right\rangle^{\dag} =  \left\{ \sum_{j=0}^{\infty} b_j x_i^{j} \; : \; 
b_{j} \in \QQ_p(\pi), \; \exists c > 0 \text{ s.t.}  
\lim_{j \rightarrow \infty} \left(\ord_p(b_j) - c \cdot j \right) \geq 0 \right\}
\end{align*}
and we define $\QQ_p(\pi)$-linear maps $\psi_p, D_i$ on $\QQ_p(\pi) \langle x_i \rangle^{\dag}$ 
by the same formulas as before. Then
\begin{align*}
\psi_p \Bigl( \exp \bigl(\pi a_i(x_i^d- x_i^{pd})\bigr) x_i^{v_i+1} \Bigr) 
    & = \sum_{\substack{m \geq 0 \text{ s.t. }\\ p \mid dm+v_i+1}}
        \biggl(\sum_{j=0}^{\floor{m/p}} \frac{(-1)^j (\pi a_i)^{m-(p-1)j}}{j! (m-pj)!} \biggr) x_i^{(dm+v_i+1)/p} \\
    & = \sum_{\; \; \; \; \; \; m,r \; \; \; \; \; \;} 
        \biggl(\; \sum_{j=0}^{r} \; \; \frac{(-1)^j (\pi a_i)^{m-(p-1)j}}{j! (m-pj)!} \biggr) x_i^{u_i+1+dr},
\end{align*}
where the outer sum on the last line is over all integers $m, r \geq 0$ 
that satisfy $p(u_i+1)-(v_i+1)=d(m-pr)$ and we have used Lemma~\ref{lem:mpr}. 
It is clear that
\begin{equation*}
x_i^{u_i+1+dr} \equiv \Bigl(\frac{u_i+1}{d}\Bigr)_r (-\pi a_i)^{-r} x_i^{u_i+1} \bmod{D_i \QQ_p(\pi) \langle x_i \rangle^{\dag} }.
\end{equation*}
Hence, it follows that
\begin{align*}
&\psi_p \Bigl( \exp \bigl(\pi a_i \bigl( x_i^d- x_i^{pd} \bigr) \bigr) x_i^{v_i+1} \Bigr) \equiv \\
&(\pi a_i)^{(p (u_i + 1) - (v_i + 1))/d} 
    \biggl( \sum_{m,r} \left( \frac{u_i+1}{d} \right)_r 
        \sum_{j=0}^{r} \frac{(p a_i^{p-1})^{r-j}}{(m-pj)!j!} \biggr) x_i^{u_i+1} \bmod{D_i \QQ_p(\pi) \langle x_i \rangle^{\dag}},
\end{align*}
where the sum is still over all integers $m, r \geq 0$  that satisfy $p(u_i+1)-(v_i+1)=d(m-pr)$. 
For $G_i \in  \QQ_p(\pi) \langle x_i \rangle^{\dag}$ with $0 \leq i \leq n$, we have
\begin{align*}
  \psi_p \biggl( \prod_{j=0}^n G_j \biggr) &= \prod_{j=0}^n \psi_p(G_j), 
& D_i \biggl( \prod_{j=0}^n G_j \biggr) &= D_i(G_i) \cdot \prod_{j \neq i} G_j.  
\end{align*}
Therefore, we can take the product of the congruences above to conclude that 
\begin{align*}
\alpha(x^{v+1}) \equiv \pi^{pk(u)-k(v)} \alpha_{u,v} \cdot x^{u+1} \bmod{\sum_{i=0}^n D_i \QQ_p(\pi) \langle x_i \rangle^{\dag}}.
\end{align*}
Since
$\mathcal{L} \; \cap \; \sum_{i=0}^n D_i \QQ_p(\pi) \langle x_i \rangle^{\dag} =  \sum_{i=0}^n D_i \mathcal{L}^{(i)}$,
we can replace $\sum_{i=0}^n D_i \QQ_p(\pi) \langle x_i \rangle^{\dag}$ 
by $\sum_{i=0}^n D_i \mathcal{L}^{(i)}$ in this expression, completing the proof.
\end{proof}

\begin{rem}
The computation in the proof of Theorem~\ref{thm:01-03-diagfrob} essentially 
goes back to Dwork~\citep[\S 4]{Dwork1964}.  Following Dwork's work, Lauder 
obtained a formula similar to ours~\citep[Proposition 10]{Lauder2004b}, that
was used both by him~\citep{Lauder2004b} and Gerkmann~\citep{Gerkmann2007} 
in their algorithms. 

However, our approach is different in two ways.  Firstly, our formula 
for~$\alpha_{u,v}$ is defined over~$\QQ_p$ instead of $\QQ_p(\pi)$. By 
eliminating the $\pi$'s, we are able to improve the complexity of
the computation of $\Phi_0$, and even of the whole deformation algorithm, 
by a factor~$p$, as we will see in Section~\ref{sec:Complexity}. Secondly, 
Lauder's formula is only valid when the $a_i$ are Teichm\"uller lifts, 
i.e.\ $a_i^{p-1}=1$. From a theoretical point of view, of course one can 
always choose the $a_i$ to be Teichm\"uller lifts, but in computations 
this is unnecessarily restrictive. 
\end{rem}

\subsection{Some estimates}

There are two remaining problems for computing the matrix~$\Phi_0$. Firstly, to 
compute~$\Phi_0$ to $p$-adic precision $N_{\Phi_0}$ using Theorem~\ref{thm:01-03-diagfrob}, 
we have to compute the elements~$\alpha_{u,v}$ to a somewhat higher precision~$N'_{\Phi_0}$ 
because of loss of precision in the computation. In Corollary~\ref{cor:NPhi0prime} we will 
obtain a bound for~$N'_{\Phi_0}$. Secondly, the sums over $m,r$ in Definition~\ref{defn:alpha} 
consist of infinitely many terms.  In Proposition~\ref{prop:MR} we will derive a finite 
expression for $\alpha_{u,v}$ to precision $N'_{\Phi_0}$. We start with some estimates that 
will be needed in the proofs.

\begin{prop} \label{prop:rfac}
For all integers $j, d, r \geq 1$ with $p \nmid d$, 
\begin{equation*}
\ord_p\Bigl(\frac{j}{d}\Bigr)_r \geq \ord_p(r!) \geq \frac{r}{p-1} - \floor{\log_p(r) + 1}.
\end{equation*}
\end{prop}

\begin{proof}
Let $s_p(r)$ denote the sum of digits in the $p$-adic expansion of~$r$ 
and observe that $s_p(r) \leq (p-1)\floor{\log_p(r) + 1}$. Then 
\begin{equation*}
\ord_p\left(\frac{j}{d}\right)_r \geq \ord_p(r!) 
    = \frac{r - s_p(r)}{p-1} \geq \frac{r}{p-1} - \floor{\log_p(r) + 1}. \qedhere
\end{equation*}
\end{proof}

\begin{prop} \label{prop:coefbound}
For all integers $m \geq 0$ and $0 \leq i \leq n$,
\begin{align*}
\ord_p \biggl( \sum_{j=0}^{\floor{m/p}} \frac{(p a_i^{p-1})^{\floor{m/p}-j}}{(m-pj)!j!} \biggr) 
    \geq \frac{p-1}{p^2}m + \floor{\frac{m}{p}}-\frac{m}{p-1}.
\end{align*}
\end{prop}
\begin{proof}
Let $\QQ_p(\pi)$ denote the totally ramified extension of $\QQ_p$ with $\pi^{p-1} = -p$ and 
extend the $p$-adic valuation such that \mbox{$\ord_p(\pi) = (p-1)^{-1}$}. Expanding
\begin{align*}
\exp \bigl( \pi a_i(x-x^p) \bigr) &= \sum_{m=0}^{\infty} \lambda_m x^m,
\end{align*}
one checks easily that, for all $m \geq 0$,
\begin{align*}
\lambda_m &= \sum_{j=0}^{\floor{m/p}} \frac{(-1)^j (\pi a_i)^{m-(p-1)j}}{(m-pj)!j!} \\
          &= (-1)^{\floor{m/p}} (\pi a_i)^{m-(p-1)\floor{m/p}} \biggl( \sum_{j=0}^{\floor{m/p}} \frac{(p a_i^{p-1})^{\floor{m/p}-j}}{(m-pj)!j!} \biggr).
\end{align*}
Therefore, it suffices to show that $\ord_p(\lambda_m) \geq m (p-1)/p^2$. This is 
well-known in the case $a_i=1$, but the proof from \citep[Lemma 4.1]{dwork1962} 
remains valid when $a_i \in \ZZ_p^{\times}$. 
\end{proof}

\begin{thm} \label{thm:alphabound} 
Let $\alpha_{u,v}$ be defined as in Definition~\ref{defn:alpha}. Then
\begin{equation*}
0 \leq \ord_p(\alpha_{u,v}) \leq \ord_p\left( \frac{(k(v)-1)!}{(k(u)-1)!} \right)+(n-k(u))+\delta,
\end{equation*}
where $\delta$ is defined as in Definition~\ref{defn:delta}. 
\end{thm}

\begin{proof}
First, we prove the left-hand inequality. By Lemma~\ref{lem:mpr}, Proposition~\ref{prop:rfac} 
and Proposition~\ref{prop:coefbound}, it suffices to show that
\begin{align} \label{eq:mp}
\frac{p-1}{p^2}m + \floor{\frac{m}{p}}-\frac{m}{p-1} + \ord_p\left(\floor{\frac{m}{p}}!\right)>-1,
\end{align}
for all integers $m \geq p$. From Proposition~\ref{prop:rfac} and the fact that $m \geq p \floor{m/p}$, it follows that
\begin{align*}
\frac{p-1}{p^2}m + \floor{\frac{m}{p}}-\frac{m}{p-1} + \ord_p\left(\floor{\frac{m}{p}}!\right) \geq \frac{p-1}{p} \floor{\frac{m}{p}} - \floor{\log_p\left(\floor{\frac{m}{p}}\right)+1}.
\end{align*}
One checks easily that the right-hand side is greater than $-1$, unless we have $p=2$ and $m=4,5$. However, in these cases~\eqref{eq:mp} still holds, as can be seen by direct
substitution.

Next, we prove the right-hand inequality. Recall from Corollary~\ref{cor:delta} that the valuations 
of the entries of the matrix~$\Phi_0$ are bounded below by $-\delta$. Thus by Theorem~\ref{thm:01-03-diagfrob}, 
we have
\begin{equation*}
-\delta \leq \ord_p\left(\frac{(k(v) - 1)!}{(k(u) - 1)!} \right)+ (n-k(u)) - \ord_p( \alpha_{u,v}),
\end{equation*}
from which the result follows.
\end{proof}

Using Theorem~\ref{thm:alphabound}, we can now bound the $p$-adic precision $N'_{\Phi_0}$ to which we
have to compute the $\alpha_{u,v}$.

\begin{cor} \label{cor:NPhi0prime}
In order to compute the matrix $\Phi_0$ to $p$-adic precision $N_{\Phi_0}$, it is sufficient to
compute the $\alpha_{u,v}$ to $p$-adic precision 
\[
N_{\Phi_0}' = N_{\Phi_0} + (n-1) + \ord_p\left((n-1)!\right)  + 2\delta,
\]
where $\delta$ is defined as in Definition~\ref{defn:delta}.
\end{cor}

\begin{proof}
Note that the loss of precision in inverting $\alpha_{u,v}$ is at most $2\ord_p(\alpha_{u,v})$.
Therefore, it follows from Theorem~\ref{thm:01-03-diagfrob} and Theorem~\ref{thm:alphabound} that
it is sufficient to compute $\alpha_{u,v}$ to precision
\begin{align*}
N_{\Phi_0} + \ord_p\left(\frac{(k(v) - 1)!}{(k(u) - 1)!} \right) + (n-k(u)) + 2\delta \leq
N_{\Phi_0} + (n-1) + \ord_p\left((n-1)!\right)  + 2\delta,
\end{align*}
where we have used that $1 \leq k(u),k(v) \leq n$.
\end{proof}

To derive a finite expression for $\alpha_{u,v}$ to precision $N'_{\Phi_0}$, 
we start with the following elementary lemma.

\begin{lem} \label{lem:log}
Given integers $\alpha, \beta \geq 2$ and defining 
$x = \beta + \log_{\alpha}(\beta) + 1$, for all real numbers $y \geq x$,
we have 
\begin{equation*}
y - \log_{\alpha}(y) \geq \beta.
\end{equation*}
\end{lem}

\begin{proof}
We first note that the function $y \mapsto y - \log_{\alpha}(y)$ is increasing 
for $y \geq 2$ because it has derivative $1 - \log_{\alpha}(e)/y > 0$.  Thus, 
it suffices to verify the result for~$x$.  Indeed, as $\beta \geq 2$ we have 
that $\log_{\alpha}(\beta) + 1 \leq \beta$, hence 
$\beta + \log_{\alpha}(\beta) + 1 \leq 2 \beta \leq \alpha \beta$,
which upon taking logarithms and rearranging yields the result.
\end{proof}

\begin{prop} \label{prop:MR}
In order to compute $\alpha_{u,v}$ to $p$-adic precision~$N'_{\Phi_0}$, 
it suffices to restrict the sums in Theorem~\ref{defn:alpha} to pairs $m,r \geq 0$ 
such that $m \leq \mathcal{M}$, or equivalently $r \leq \mathcal{R}$, where 
\begin{align*}
\mathcal{M} &= \ceil{\frac{p^2}{p-1}(N'_{\Phi_0}+\log_p(N'_{\Phi_0}+3)+4)} - 1,
&\mathcal{R} &= \floor{\mathcal{M}/p}.
\end{align*}
\end{prop}

\begin{proof}
By Proposition~\ref{prop:rfac} and Proposition~\ref{prop:coefbound}, we have
\begin{align*}
\ord_p \biggl( \Bigl( \frac{u_i+1}{d} \Bigr)_r \sum_{j=0}^{r} \frac{(p a_i^{p-1})^{r-j}}{(m-pj)!j!} \biggr) 
    &\geq  \frac{p-1}{p^2}m + \floor{\frac{m}{p}}-\frac{m}{p-1}+\frac{r}{p-1}-\floor{\log_p(r)+1} \\
    &\geq  \frac{p-1}{p^2}m - \log_p(m)-1.
\end{align*}
Therefore, it is sufficient to restrict the sums in Definition~\ref{defn:alpha} 
to $m,r \geq 0$ for which
\begin{align*}
\frac{p-1}{p^2}m - \log_p(m) -1 < N'_{\Phi_0}. 
\end{align*}
We can now apply Lemma~\ref{lem:log} with $y=m(p-1)/p^2$, $\alpha=p$ 
and $\beta=N'_{\Phi_0}+3$ to obtain the result.
\end{proof}

Finally, we formalise the procedure for computing the entries of~$\Phi_0$ to 
$p$-adic precision~$N_{\Phi_0}$ in Algorithm~\ref{alg:Diagfrob}.

\begin{algorithm}
\caption{Compute the matrix $\Phi_0$.}
\label{alg:Diagfrob}
\begin{algorithmic}
\vspace{1mm}
\Require $P_0=a_0 x_0^d + \dotsb + a_n x_n^d$ 
         with $a_0,\dotsc,a_n \in \ZZ_p^{\times}$, 
         $p$-adic precision $N_{\Phi_0} \geq 0$.
\Ensure  The matrix $\Phi_0$ for the action of $p^{-1} \Frob_p$ 
         on $\Hrig^n(U_0)$ with respect to basis $\cB$ to $p$-adic 
         precision $N_{\Phi_0}$.
\Procedure{DiagFrob}{$P_0,N_{\Phi_0}$}
\State \begin{compactenum}[{\hspace{1em}} 1.] \vspace{-1.24em}
\item Determine $N'_{\Phi_0}$ from Corollary~\ref{cor:NPhi0prime}. 
\item Let $\Phi_0 \in M_{b \times b}(\QQ_p)$ be the zero matrix.
\item[] \textbf{for} $x^u \in B$ \textbf{do} 
\item[] \begin{compactenum}[{\hspace{1em}} 1.]
        \item Determine the unique $x^v \in B$ such that $v_i = p (u_i + 1) - 1 \bmod{d}$.
        \item Compute $\alpha_{u,v}$ to $p$-adic 
              precision $N'_{\Phi_0}$ using Proposition~\ref{prop:MR}.
        \item Set $(\Phi_0)_{u,v} \gets (-1)^{k(v)} \left( \frac{ (k(v)-1)!}{(k(u)-1)!} \right) p^{n-k(u)} \alpha_{u,v}^{-1} \bmod{p^{N_{\Phi_0}}}$.
      \end{compactenum}   
 \item \textbf{return} $\Phi_0$      
\end{compactenum}
\EndProcedure
\end{algorithmic}
\end{algorithm}

\begin{assump} \label{assump:diag}
From now on we assume that our family of 
hypersurfaces~$\mathcal{X}/\mathcal{S}$ is defined by a 
polynomial $P \in \ZZ_q[t][x_0,\dotsc,x_n]$ for which 
$P(0) \in \ZZ_q[x_0,\dotsc,x_n]$ is of the form 
$P_0=a_0 x_0^d + \dotsb + a_n x_n^d$ with $a_0,\dotsc,a_n \in \ZZ_p^{\times}$, 
so that we can apply Algorithm~\ref{alg:Diagfrob}. 
\end{assump}

\begin{rem}
Note that Assumption~\ref{assump:diag} 
also implies that $\mathcal{S}$ can be chosen to satisfy Assumption~\ref{assump:S} 
and Assumption \ref{assump:R}, since $P_0$ defines a smooth hypersurface and 
we have that $R(0) \in \ZZ_p^{\times}$. 
\end{rem}


\section{Solving the differential equation}
\label{sec:DifferentialSystem}

In this section we explain how to solve the $p$-adic differential 
equation for the horizontal sections of the Gauss--Manin 
connection $\nabla$, in order to obtain a local expansion of the 
matrix for the action of $p^{-1} \Frob_p$ on $\Hrig^{n}(U/S)$.  

All our matrices will be defined with respect to the basis $\cB$. Recall 
that $M \in M_{b \times b}(\QQ_q(t))$ denotes 
the matrix for the Gauss--Manin connection $\nabla$ on 
$\HdR^n(\mathfrak{U}/\mathfrak{S})$ and 
$\Phi \in M_{b \times b} \bigl(\QQ_q \langle t,1/r \rangle^{\dag} \bigr)$ 
the matrix for the $\sigma$-semilinear action of~$p^{-1} \Frob_p$ 
on $\Hrig^{n}(U/S)$, where $\sigma$ is defined as 
in Definition~\ref{defn:sigma}.

As we saw in Section~\ref{sec:Background}, these matrices satisfy 
the differential equation
\begin{align} \label{eq:Phi}
\left(\frac{d}{dt} + M\right) \Phi &= p t^{p-1} \Phi \sigma(M), &\Phi(0)& = \Phi_0, 
\end{align}
where $\Phi_0 \in M_{b \times b}(\QQ_p)$ is the matrix for the action 
of $p^{-1} \Frob_p$ on $\Hrig^n(U_0)$. Our goal is the computation of 
the power series expansion of~$\Phi$ at $t=0$ to $t$-adic precision~$K$ 
and $p$-adic precision $N_{\Phi}$, i.e.\ as an element of 
$M_{b \times b}(\QQ_q[[t]])$ modulo $t^K$ and $p^{N_{\Phi}}$.

We first observe that if 
$C \in M_{b \times b}(\QQ_q[[t]])$ denotes 
the unique solution to the differential equation
\begin{align} \label{eq:01-GMDE-Homogenous}
\Bigl(\frac{d}{dt} + M\Bigr) C &= 0, &C(0)&=I, 
\end{align}
where $I$ denotes the identity matrix, 
then the matrix $\Phi = C \Phi_0 \sigma(C)^{-1}$ satisfies 
Equation~\eqref{eq:Phi}. So it is sufficient to solve 
Equation~\eqref{eq:01-GMDE-Homogenous}. 
We now give a bound on the rate of 
convergence of $C=\sum_{i=0}^{\infty} C_i t^i$ that
follows from recent work of 
Kedlaya~\citep{Kedlaya2010}. 
We let $\delta$ be defined as in Definition~\ref{defn:delta}. 

\begin{thm} \label{thm:valC}
For all $i \geq 1$, we have
\begin{equation*}
\ord_p(C_i) \geq - \left(2 \delta + (n - 1)\right) \ceil{\log_p(i)}.
\end{equation*}
\end{thm}

\begin{proof}
It follows from~\citep[Theorem~{18.3.3}]{Kedlaya2010} that
\begin{equation*}
\ord_p(C_i) \geq \left( \ord_p(\Phi) + \ord_p(\Phi^{-1}) \right) \ceil{\log_p(i)},
\end{equation*}
but from Corollary~\ref{cor:delta} we already know that $\ord_p(\Phi) \geq -\delta$ and 
$\ord_p(\Phi^{-1}) \geq -\delta-(n-1)$.
\end{proof}

\begin{rem}
In~\citep[Remark~18.3.4]{Kedlaya2010} Kedlaya also includes the bound
\begin{equation*}
\ord_p(C_i) \geq (b - 1) \ord_p(M) 
            + \left( \ord_p(\Phi) + \ord_p(\Phi^{-1}) \right) \floor{\log_p(i)},
\end{equation*}
which can sometimes be used to improve Theorem \ref{thm:valC} 
slightly, for example when $\ord_p(M)$ is nonnegative.
\end{rem}

\begin{rem} \label{rem:Cinv}
The bound from Theorem~\ref{thm:valC} 
also applies to the inverse matrix~$C^{-1}$, as this matrix satisfies 
the dual differential equation 
\begin{align} \label{eq:01-GMDE-Dual}
\left(\frac{d}{dt} - M^t\right) \left(C^{-1}\right)^t &= 0, &C^{-1}(0)& = I,
\end{align}
that carries a Frobenius structure 
given by the matrix~$(\Phi^{-1})^t$.
\end{rem}

We only know the matrix $M$ to some finite $p$-adic precision 
$N_M$, and we need to compute $C$ to some finite $t$-adic and 
$p$-adic precisions $K,N_C$, respectively. The following result gives an 
expression for $N_M$ in terms of $K,N_C$. For a matrix 
$A=\sum_{i=0}^{\infty} A_i t^i$ we write 
$\overline{A}=\sum_{i=0}^{K-1} A_i t^i$ 
in what follows. 

\begin{prop} \label{prop:N_M}
Let $K,N_C \in \NN$ and define
\[
N_M= N_C + (2 \delta + n) \ceil{\log_p(K-1)}+1.
\]
Let $\tilde{M} \in M_{b \times b}(\QQ_q(t))$ be
an approximation of $M$ to $p$-adic precision $N_M$, i.e.
such that $\ord_p(\tilde{M}-M) \geq N_M$, and suppose that 
$\tilde{C}=\sum_{i=0}^{\infty} \tilde{C}_i t^i$ satisfies the
differential equation
\begin{align*}
\left(\frac{d}{dt}+\tilde{M} \right) \tilde{C}&=0, &\tilde{C}(0)=I&.
\end{align*}
Then $\ord_p(\tilde{C}_i-C_i) \geq N_C$ for all $i < K$.
\end{prop}

\begin{proof}
From the expressions
\begin{align*}
C(t)&=\exp\left(- \int_{0}^t M(s) ds\right), &
\tilde{C}(t)&=\exp\left(-\int_{0}^t \tilde{M}(s) ds\right),
\end{align*}
it follows that
\begin{equation} \label{eq:CtildeminusC}
\tilde{C}(t) - C(t) = C(t) \left( \exp \left( \int_{0}^t \left( \tilde{M}(s)-M(s) \right) ds \right)- I \right).
\end{equation}
Since $\ord_p(\tilde{M}-M) \geq N_M$, we obtain
\begin{equation*}
\ord_p \left( \overline{\frac{1}{i!} \left(\int_{0}^t \left( \tilde{M}(s)-M(s) \right) ds \right)^i } \right) \geq 
N_C + (2 \delta + n-1) \ceil{\log_p(K-1)}
\end{equation*}
for all $1 \leq i \leq K-1$, where we have used that 
$\ord_p(i!) \leq i/(p-1) \leq 1$. Moreover, from 
Theorem~\ref{thm:valC}, we already know that 
\[
\ord_p(\overline{C}) \geq -(2 \delta + n-1) \ceil{\log_p(K-1)}.
\] 
From these two inequalities and Equation~\eqref{eq:CtildeminusC}, we 
deduce $\ord_p(\overline{\tilde{C}-C}) \geq N_C$. 
\end{proof}

Now we explain how to compute the solution $C$ to 
Equation~\eqref{eq:01-GMDE-Homogenous} to $p$-adic precision~$N_C$ and 
$t$-adic precision~$K$, assuming that the connection matrix~$M$ has been 
computed to $p$-adic precision~$N_M$ as defined in Proposition~\ref{prop:N_M}. 
We can write $M = G/r$, with 
$G = \sum_{i=0}^{\deg(G)} G_i t^i \in M_{b \times b}(\QQ_q[t])$ 
and $r = \sum_{i=0}^{\deg(r)} r_i t^i \in \ZZ_q[t]$ a divisor of the 
polynomial~$R$ defined in Definition~\ref{defn:resultant}. 
Note that the degree of~$r$ might be smaller than the degree 
of~$R$, which will speed up our computations. By Assumption~\ref{assump:diag}, 
we have $r(0) \neq 0 \bmod{p}$, so in particular $r(0) \neq 0$.  

We can obtain a power series solution $C = \sum_{i=0}^{\infty} C_i t^i$ to
the equation
\begin{align*}
r \frac{dC}{dt} + G C &= 0, &C(0)& = I,
\end{align*}
which is clearly equivalent to 
Equation~\eqref{eq:01-GMDE-Homogenous}, using the following recursion:  
\begin{align} \label{eq:recursiondifeq}
C_0 &= I, \nonumber \\
C_{i+1} &= \frac{-1}{r_0 (i+1)} \left(
    \sum_{j=\max{\{0,i-\deg(G)\}}}^i G_{i-j} C_j + 
    \sum_{j=\max{\{0,i-\deg(r)\}}+1}^i r_{i-j+1} (j C_j) \right).
\end{align}
Again we will only carry out this computation to some finite $p$-adic 
working precision~$N'_C$, and if we want $C$ to be correct to $p$-adic 
precision~$N_C$, then the precision~$N'_C$ has to be somewhat higher 
because of error propagation. An expression for~$N'_C$, in terms of~$N_C$ 
and the desired $t$-adic precision~$K$, was given by 
Lauder~\citep[Theorem~5.1]{Lauder2006}, but his result can be
significantly improved using Theorem~\ref{thm:valC}, as we will now show. 

Let $\tilde{C}=\sum_{i=0}^{\infty} \tilde{C}_i t^i$ denote an 
approximation to $C$ obtained using the approximate recursion:  
\begin{align*}
\tilde{C}_0 &= I, \\
\tilde{C}_{i+1} &= \frac{-1}{r_0 (i+1)} \left(
    \sum_{j=\max{\{0,i-\deg({G})\}}}^i {G}_{i-j} \tilde{C}_j + 
    \sum_{j=\max{\{0,i-\deg({r})\}}+1}^i {r}_{i-j+1} (j \tilde{C}_j) \right) + {E}_{i+1},
\end{align*}
where $\ord_p({E}_i) \geq N'_{C}$ for all $i \geq 1$, 
so that the matrices~$\tilde{C}_i$ are computed with $p$-adic working 
precision~$N'_C$.

\begin{prop} \label{thm:errorprop}
Let $K,N_{C} \in \NN$, and suppose that
\begin{align*}
N'_C          &= N_{C} + \left(2 \left(2 \delta + (n-1)\right) + 1\right) \ceil{\log_p(K-1)}.
\end{align*} 
Then $\ord_p(\tilde{C}_i-C_i) \geq N_{C}$ for all $i < K$.
\end{prop}

\begin{proof}
The matrix $\tilde{C}$ satisfies the differential equation
\begin{align*}
\frac{d\tilde{C}}{dt}+M \tilde{C}&=E, &\tilde{C}(0)&=I,
\end{align*}
where we have denoted $E=\sum_{i=1}^{\infty} E_i t^i$.
One also checks that the matrix $C^{-1}\tilde{C}$ satisfies 
the differential equation
\begin{align*}
\frac{d(C^{-1}\tilde{C})}{dt} &=C^{-1} E, &(C^{-1}\tilde{C})(0)&=I,
\end{align*}
from which it follows that
\begin{equation} \label{eq:integral}
\tilde{C}(t)-C(t) = C(C^{-1} \tilde{C}-I) = C(t) \left(\int_{0}^t C^{-1}(s) E(s) ds \right).
\end{equation}
We know from Theorem~\ref{thm:valC} and Remark~\ref{rem:Cinv} that
\begin{equation} \label{eq:boundCCinv}
\ord_p(\overline{C}),\ord_p(\overline{C^{-1}}) \geq 
-(2 \delta + n-1) \ceil{\log_p(K-1)},
\end{equation}
and we hence obtain
\[
\ord_p \left(\overline{\int_{0}^t C^{-1}(s) E(s) ds }\right) \geq 
N_{C} + \left( \bigl(2 \delta + (n-1)\bigr) \right) \ceil{\log_p(K-1)}.
\]
From the bounds~\eqref{eq:boundCCinv} and Equation~\eqref{eq:integral}, 
we deduce that $\ord_p\left(\overline{\tilde{C}-C}\right) \geq N_C$.
\end{proof}

\begin{rem}
A result similar to Proposition~\ref{thm:errorprop} with a larger constant in front of 
the logarithm was obtained by Lauder in \citep[Theorem 5.1]{Lauder2006}. We have
not been able to find something similar to Proposition~\ref{prop:N_M} in Lauder's work.
\end{rem}

\begin{rem} \label{rem:sigmatrick}
In order to determine the power series expansion of the matrix~$\Phi$, 
we also need to compute the matrix $\sigma(C)^{-1}$. We could compute 
the matrix~$C^{-1}$ using matrix inversion over the ring $\mathbf{Q}_q[[t]]$. 
However, solving~\eqref{eq:01-GMDE-Dual} turns
out to be more efficient. 
\end{rem}

We finally give all the precisions necessary for computing the power series 
expansion of $\Phi$ at $t=0$ to $t$-adic precision $K$ and $p$-adic 
precision $N_{\Phi}$.

\begin{thm} \label{thm:Ni}
Let $K,N_{\Phi} \in \NN$ and define:
\begin{eqnarray*}
N_{\Phi_0}   		&=& N_{\Phi}+\left(2\delta+(n-1)\right) \left( \ceil{\log_p(K-1)} + \ceil{\log_p(\ceil{K/p}-1)}\right),\\
N_{C}				&=& N_{\Phi}+\left(2\delta+(n-1)\right) \ceil{ \log_p(\ceil{K/p}-1)} + \delta, \\
N_{C^{-1}}			&=& N_{\Phi}+\left(2\delta+(n-1)\right) \ceil{\log_p(K-1)} + \delta, \\
N_M                 &=& N_{\Phi}+\left(2\left(2 \delta + (n-1)\right) + 1\right) \ceil{\log_p(K-1)}+1, \\
N'_C			    &=& N_{\Phi}+\left(2 \left(2 \delta + (n-1)\right) + 1\right) \ceil{\log_p(K-1)}, \\
N'_{C^{-1}}	        &=& N_{\Phi}+\left(2 \left(2 \delta + (n-1)\right) + 1\right) \ceil{\log_p(\ceil{K/p}-1)}.
\end{eqnarray*}

In order to compute the power series expansion of the matrix~$\Phi$ at $t=0$ 
with $t$-adic precision~$K$ and $p$-adic precision~$N_{\Phi}$, it is 
sufficient to compute the matrix~$\Phi_0$ to $p$-adic precision~$N_{\Phi_0}$,
the matrix~$C$ to $t$-adic precision~$K$ and $p$-adic precision~$N_{C}$, 
the matrix~$C^{-1}$ to $t$-adic precision $\ceil{K/p}$ and $p$-adic 
precision~$N_{C^{-1}}$, and the matrix~$M$ to $p$-adic precision~$N_M$.

Therefore, while solving Equation~\eqref{eq:01-GMDE-Homogenous} for $C$ and 
Equation~\eqref{eq:01-GMDE-Dual} for $C^{-1}$, using a recursion like in 
Equation~\eqref{eq:recursiondifeq}, it is sufficient to use $p$-adic 
working precisions $N'_C$ and $N'_{C^{-1}}$, respectively.
\end{thm}

\begin{proof}
Recall that $\Phi = C \Phi_0 \sigma(C)^{-1}$. The sufficient $t$-adic 
precisions are clear. We can apply Proposition~\ref{prop:matrixproductval}, 
using Theorem~\ref{thm:valC} for both $C$ and $C^{-1}$ and the fact that 
$\ord_p(\Phi_0) \geq -\delta$ from Corollary~\ref{cor:delta}, to obtain 
the sufficient $p$-adic precisions for the matrices $\Phi_0$, $C$ and $C^{-1}$. 
The sufficient precision for the matrix $M$ follows from Proposition~\ref{prop:N_M},
and the sufficient working precisions $N'_C$ and $N'_{C^{-1}}$
follow from Proposition~\ref{thm:errorprop}.
\end{proof}

Now we have all the ingredients to compute the power series expansion of 
$\Phi$ at $t=0$ to any given $p$-adic and $t$-adic precisions, as formalised 
in Algorithm~\ref{alg:expansion}.

\begin{algorithm}
\caption{Compute the power series expansion of $\Phi$ at $t=0$.}
\label{alg:expansion}
\begin{algorithmic}
\vspace{1mm}
\Require $P \in \ZZ_q[t][x_0,\dotsc,x_n]$ satisfying Assumption~\ref{assump:diag}, $t$-adic precision~$K$, $p$-adic precision~$N_{\Phi}$.
\Ensure  The power series expansion of $\Phi$ at $t=0$ to $t$-adic precision $K$ and $p$-adic precision $N_{\Phi}$.
\Procedure{FrobSeriesExpansion}{$N_{\Phi},K$} 
\State \begin{compactenum}[{\hspace{6pt}} 1.] \vspace{-1.24em}
\item Determine $N_M,N_{\Phi_0},N_C,N_{C^{-1}},N'_C$, and $N'_{C^{-1}}$ from Theorem~\ref{thm:Ni}.
\item $M \gets \textsc{GMConnection}(P,N_M)$
\item $\Phi_0 \gets \textsc{DiagFrob}(P_0,N_{\Phi_0})$
\item Solve Equation~\eqref{eq:01-GMDE-Homogenous} for $C$ to $t$-adic precision $K$ and $p$-adic precision $N_{C}$:
\begin{compactenum}[a.] 
\item[] $C_0 \gets I$
\item[] \textbf{for} $i=0$ \textbf{to} $K-2$ \textbf{do} 
\item[] \hspace{0.6em} $C_{i+1} \gets \frac{-1}{r_0 (i+1)} \Bigl(\sum_j G_{i-j} C_j + \sum_j r_{i-j+1} (j C_j) \Bigr) \bmod{p^{N'_C}}$
\item[] $C \gets \sum_{i=0}^{K-1} C_i t^i \bmod{p^{N_C}}$
\end{compactenum}
\item Solve Equation~\eqref{eq:01-GMDE-Dual} for $C^{-1}$ to $t$-adic precision $\ceil{K/p}$ and $p$-adic precision $N_{C^{-1}}$:
\begin{compactenum}[a.]
\item[] $(C^{-1})_0 \gets I$
\item[] \textbf{for} $i=0$ \textbf{to} $\ceil{K/p}-2$ \textbf{do}
\item[] \hspace{0.6em} $(C^{-1})_{i+1} \gets  \frac{-1}{r_0(i+1)} \Bigl(\sum_j -G_{i-j}^t (C^{-1})_j + \sum_j r_{i-j+1} (j (C^{-1})_j) \Bigr) \bmod{p^{N'_{C^{-1}}}}$
\item[] $C^{-1} \gets \sum_{i=0}^{\ceil{K/p}-1} (C^{-1})_i t^{i} \bmod{p^{N_{C^{-1}}}}$
\end{compactenum}
\item $\Phi \gets C \Phi_0 \sigma(C^{-1})$
\item \Return $\Phi$
\end{compactenum}
\EndProcedure
\end{algorithmic}
\end{algorithm}


\section{The zeta function of a fibre}

\label{sec:ZetaFunctions}

\subsection{Evaluating $\Phi$ at a point}

In the previous section we described how to compute the power series 
expansion at $t=0$ of the matrix~$\Phi$ for the action of $p^{-1} \Frob_p$ on 
$\Hrig^{n}(U/S)$. We now want to evaluate $\Phi$ at the Teichm\"uller 
lift~$\hat{\tau}$ of some nonzero $\tau \in S(\FF_{\mathfrak{q}})$, 
where $\FF_{\mathfrak{q}}/\FF_q$ is a finite field extension, but 
the power series expansion of $\Phi$ 
at $t=0$ usually only converges on the open unit disc $\abs{t} < 1$ 
and hence cannot be evaluated at $\hat{\tau}$. 

However, since $\Phi$ is a matrix of overconvergent functions, i.e.\ 
\[
\Phi \in M_{b \times b}\left(\QQ_q \left\langle t,1/r \right\rangle^{\dag}\right),
\]
it can be approximated to any given $p$-adic precision $N$ by a matrix 
of rational functions, which can then be evaluated at $\hat{\tau}$. To 
convert the power series expansion for $\Phi$ to such an approximation 
by rational functions, we need a bound on the pole order of these 
rational functions at their poles, as a function of~$N$.

The bounds used by Lauder~\citep[\S 8.1]{Lauder2004a} and 
Gerkmann~\citep[\S 6]{Gerkmann2007} were not very sharp, 
and this significantly impacted their algorithms.
Recently, under some small additional assumptions, Kedlaya and the second 
author~\citep[Theorem~2.1]{KedlayaTuitman2012} obtained a sharper bound 
that we state below in a simplified form. For $z \in \mathbf{P}^1(\bar{\QQ}_q)$, we
denote the valuation on $\QQ_q(t)$ corresponding to $z$ by $\ord_z(-)$, and
we extend $\ord_z(-)$ to polynomials and matrices over $\QQ_q(t)$ by
taking the minimum over the coefficients and entries, respectively.
We recall that $M$ denotes the matrix of the Gauss--Manin connection~$\nabla$
with respect to the basis~$\cB$ from Definition~\ref{defn:MonBasis} 
and let $\delta$ be defined as in Definition~\ref{defn:delta}.

\begin{thm} \label{thm:KedlayaTuitman}
Let $D \subset \mathbf{P}^1(\bar{\QQ}_q)$ be a residue disk and
$z \in D$ a point such that $M$ has at most a simple pole at~$z$ 
and no other poles contained in $D$. Suppose that 
the exponents $\lambda_1, \dotsc, \lambda_{b}$ of $M$ at $z$, which are defined 
as the eigenvalues  of the matrix $(t - z) M \vert_{t=z}$ and 
known to be rational numbers, are contained in $\ZZ_p \cap \QQ$. 
For $i \in \NN$, let 
\begin{align*}
f(i) = \max \left\{ -\left(2 \delta + (n-1) \right) \lceil \log_p(i) \rceil, 
(b-1) \ord_p(M) -\left(2 \delta + (n-1) \right) \lfloor \log_p(i) \rfloor 
\right\},
\end{align*}
and define 
\begin{align*}
c = \begin{cases}
0 & \mbox{if $\ord_p(M) \geq 0$}, \\
\min\{0, i + f(i): i \in \NN\} & \mbox{if $\ord_p(M) < 0$}.
\end{cases}
\end{align*}
For $N \in \NN$, let 
\begin{align*}
g(N) &= \max \{i \in \NN : i + f(i) - \delta + c  < N \},
\end{align*}
and define
\begin{align*}
\alpha_1    &= \lfloor -p \min_i \{ \lambda_i \} + \max_{i} \{\lambda_i\} \rfloor, \\ 
\alpha_2    &=  \left \{ 
         \begin{array}{cl}
         0  & \mbox{if $M$ does not have a pole at $z$},  \\
         0  & \mbox{if $z \in \{0,\infty \}$}, \\
         g(N) & \mbox{otherwise}.
         \end{array}
         \right.
\end{align*}
Then the matrix $\Phi$ is congruent 
modulo $p^{N}$ to a matrix $\tilde{\Phi} \in M_{b \times b}(\QQ_q(t))$ for which
\begin{equation*}
\ord_z\left(\tilde{\Phi}\right) \geq -(\alpha_1+p \alpha_2)
\end{equation*}
and $\tilde{\Phi}$ has no other poles contained in $D$.
\end{thm}

\begin{proof}
Since the matrix $\Phi$ defines a Frobenius structure on the vector 
bundle $\Hrig^n(U/S)$ with connection $\nabla$, we can 
apply~\citep[Theorem~2.1]{KedlayaTuitman2012}. We have 
replaced $\ord_p\left(\Phi\right)$ and $\ord_p\left(\Phi\right)+\ord_p\left(\Phi^{-1}\right)$ 
by their respective lower bounds $-\delta$ and $-2 \delta + (n-1)$, since 
we might not know them exactly a priori.
\end{proof}

\begin{rem}
In practice, the constants $-(2 \delta + (n-1))$, $c$, $\alpha_1$, and $-\delta$ 
are very small in absolute value, so that $g(N)$ is about $N$ and 
the lower bound for the order of~$\Phi$ modulo $p^N$ at the point $z$ is 
roughly $-pN$.
\end{rem}

\begin{rem} \label{rem:basischange}
The condition that $M$ has a simple pole at $z$ is not a serious restriction. 
By \citep[Theorem 2.1]{Lauder2011}, one can always find a matrix $W \in GL_{b}(\QQ_q[t,1/(t-z)])$
such that the connection matrix has a simple pole at $z$ with respect to the basis 
$[w_1, \dotsc, w_b]$ defined by $w_j = \sum_{i=1}^b W_{ij} e_i$ for all 
$1 \leq j \leq b$, where $\cB=[e_1,\dotsc,e_b]$ denotes our basis from 
Definition~\ref{defn:MonBasis}. Now by \citep[Corollary 2.6]{KedlayaTuitman2012}, 
the matrix $\Phi$ is congruent modulo $p^{N+\ord_p(W)+\ord_p(W^{-1})}$ 
to a matrix $\tilde{\Phi} \in M_{b \times b}(\QQ_q(t))$ for which
\[
\ord_z\bigl(\tilde{\Phi}\bigr) \geq -(\alpha_1+p \alpha_2)+\ord_z(W)+p\ord_z\bigl(W^{-1}\bigr)
\]
and $\tilde{\Phi}$ has no other poles contained in $D$.
Here, $\alpha_1$ and $\alpha_2$ are defined as in 
Theorem~\ref{thm:KedlayaTuitman}, but with $\delta$ replaced by
\begin{equation*}
\delta' = \delta-\left(\ord_p(W)+\ord_p\left(W^{-1}\right)\right), 
\end{equation*}
since for the matrix $\Phi'$ of $p^{-1}\Frob_p$ with respect to the basis 
$[w_1,\dotsc,w_b]$, in general we only have that $\ord_p\left(\Phi'\right) \geq -\delta'$. 
\end{rem}

When one of the other conditions in Theorem~\ref{thm:KedlayaTuitman} is not satisfied, 
we can use the alternative bound below. Recall that our family of 
smooth projective hypersurfaces $\mathcal{X}/\mathcal{S}$ is defined by the homogeneous 
polynomial $P \in \ZZ_q[t][x_0,\dotsc,x_n]$.

\begin{thm} \label{thm:Gerkmann}
Define the matrices $\Delta_k$ over $\QQ_q(t)$ as in 
Definition~\ref{defn:Deltak}.
Let $D \subset \mathbf{P}^1(\bar{\QQ}_q)$ be a residue disk
and for any point $z \in D$ put
\begin{eqnarray*}
\mu_z &=& \sum_{k=2}^n \ord_z\left(\Delta_k^{-1}\right), \\
\nu_z &=& \ord_z\left(\Delta_{n+1}^{-1}\right).
\end{eqnarray*}
For $N \in \NN$, define 
\begin{align*}
h(N) = \max \left\{ i \in \NN : i+(n-1)+\ord_p ((n-1)!)-n \floor{\log_p(p(n+i)-1)} < N \right\},
\end{align*}
and put
\[
\beta_z = \begin{cases}
-\mu_z -(p(n+h(N))-n) \nu_z & \mbox{if $z \neq \infty$}, \\
-\mu_z -(p(n+h(N))-n) \nu_z + p h(N) \deg_t(P) & \mbox{if $z = \infty$}.
\end{cases}
\]
Then $\Phi$ is congruent modulo $p^N$ to a matrix 
$\tilde{\Phi} \in M_{b \times b}(\QQ_q(t))$ for which
\begin{equation*}
\ord_z\bigl(\tilde{\Phi}\bigr) \geq -\beta_z,
\end{equation*}
for all points $z \in D$.
\end{thm}

\begin{proof}
First we extend $\sigma$ to the ring of overconvergent functions:
\begin{align*}
A^{\dag} =
&\Bigl\{ \sum_{i_0,\dotsc,i_{n+3}=0}^{\infty} a_{i_0, \dotsc, i_{n+3}} 
\frac{x_0^{i_0} \dotsm x_n^{i_{n}}t^{i_{n+1}}}{r^{i_{n+2}} P^{i_{n+3}}} \; : \;
a_{i_0, \dotsc i_{n+3}} \in \QQ_q, \; \\ 
&\exists c > 0 \text{ s.t.} \lim_{i_0+\dotsb+i_{n+3} \rightarrow \infty} \bigl(\ord_p(a_{i_0, \dotsc, i_{n+3}} ) - 
c(i_0+\dotsb+i_{n+3})\bigr) \geq 0 \Bigr\},
\end{align*}
by putting $\sigma(x_i) = x_i^p$ for $0 \leq i \leq n$. Note that
\begin{equation*}
\sigma\left(\frac{1}{P}\right) = 
    P^{-p} \left( 1-\frac{P^p-\sigma(P)}{P^p} \right)^{-1}
\end{equation*}
is an element of $A^{\dag}$ since 
\begin{equation*}
\ord_p\left(P^p-\sigma(P)\right) \geq 1.
\end{equation*}
We use the notation and terminology from Section~\ref{sec:Connection}.
It is known that the de Rham cohomology of 
$A^{\dag}/\QQ_q \langle t,1/r \rangle^{\dag}$ 
is isomorphic to the rigid cohomology of $U/S$, and that the action of 
$F_p$ on the rigid cohomology of $U/S$ is induced by $\sigma$. Recall
that our basis vectors for $\Hrig(U/S)$ 
are of the form $x^u \Omega / P^{\ell}$ with $x^u \in B_{\ell}$ and 
$\ell \leq n$.  We observe that
\begin{equation*}
p^{-1} \Frob_p \left(\frac{x^u \Omega}{P^{\ell}} \right) \equiv
\sum_{i=0}^{\infty} \eta_i p^{n-1+i} \left(\frac{P^p-\sigma(P)}{p} \right)^i 
(x_0 \dotsm x_n)^{p-1} \left( \frac{x^{pu} \Omega}{P^{p(\ell+i)}} \right)
\end{equation*}
in $\Hrig^n(U/S)$, where $\eta_i \in \NN$ is defined by the equality 
$(1-y)^{-\ell} = \sum_{i=0}^{\infty} \eta_i y^i$.
Using Algorithm~\ref{alg:PoleRed}, we can write
\begin{equation} \label{eqn:Froblift}
\eta_i p^{n-1+i} \left(\frac{P^p-\sigma(P)}{p} \right)^i 
(x_0 \dotsm x_n)^{p-1} \left( \frac{x^{pu} \Omega}{P^{p(\ell+i)}} \right) \equiv
\frac{\gamma_{i,1} \Omega}{P^1}+\dotsb+\frac{\gamma_{i,n} \Omega}{P^n},
\end{equation}
where $\gamma_{i,j}$ is contained in the $L$-span of $B_j$ for all $i \geq 0$ and 
$0 \leq j \leq n$. 

Actually, we can slightly modify 
Algorithm~\ref{alg:Decompose}, so that we only have to solve systems of the 
form $\Delta_{n+1} v = w$ as long as the pole order is at least~$n+1$.  
Indeed, we may assume, without loss of generality, that $Q=x^v$ is a monomial, 
and can find some other monomial $x^w$ such that $x^{v-w}$ is of degree 
$(n+1)d-(n+1)$. Now we apply Algorithm~\ref{alg:Decompose} to $x^{v-w}$, 
noting that $B_{n+1}=\emptyset$, and multiply the output by~$x^w$ to obtain 
the decomposition of~$x^v$.

Note that the left-hand side of Equation~\eqref{eqn:Froblift} does not have a 
pole at $z \neq \infty$, and has order at least $-i p \deg_t(P)$ at $z=\infty$. 
When applying Algorithm~\ref{alg:PoleRed} to this expression, as long as the 
pole order is at least~$n+1$, the order at~$z$ drops by at most $\nu_z$ in every 
reduction step, and it drops by at most $\mu_z$ in the remaining reduction steps. 
This implies that
\begin{align} \label{al:ordz}
\ord_z(\gamma_{i,j}) &\geq 
\begin{cases}
\mu_z + (p(n+i)-n) \nu_z                 &\mbox{if } z \neq \infty, \\
\mu_z + (p(n+i)-n) \nu_z - pi \deg_t(P)  &\mbox{if } z=\infty.
\end{cases}
\end{align}

For every $\tau \in S(\bar{\FF}_q)$, we
define lattices $\Lambda_{\tau,crys}$, $\Lambda_{\tau,\cB}$ in $\Hrig^n(U_{\tau})$
as in the proof of Theorem~\ref{thm:deltabound}. It follows from
\citep[Proposition 3.4.6]{AbbottKedlayaRoe2006}, Equation~\eqref{eqn:Froblift}, 
and the inclusion $(n-1)!\Lambda_{\tau,crys} \subset \Lambda_{\tau,\cB}$ that
\begin{equation*}
\frac{\gamma_{i,1} \Omega}{P^1}+\dotsb+\frac{\gamma_{i,n} \Omega}{P^n} \subset 
p^{i+(n-1)-\ord_p((n-1)!)-n \floor{\log_p(p(n+i)-1)}} \Lambda_{\tau,\cB},
\end{equation*}
which implies that
\begin{equation*}
\ord_p(\gamma_{i,j}) \geq i+(n-1) - \ord_p((n-1)!) - n \floor{\log_p(p(n+i)-1)}.
\end{equation*}
The entries of $\Phi$ are coefficients of sums of the form 
$\sum_{i=0}^{\infty} \gamma_{i,j}$. Now modulo $p^N$, we can 
restrict this sum to terms for which $\ord_p(\gamma_{i,j})<N$,
and we can compute the order at $z$ of these terms using
\eqref{al:ordz}. This completes the proof.
\end{proof}

\begin{rem}
A result similar to, but weaker than, Theorem \ref{thm:Gerkmann} 
was obtained by Gerkmann in \citep[Section 6]{Gerkmann2007}.
\end{rem}

\begin{thm} \label{thm:defD}
Let $N_{\Phi} \in \NN$. We can explicitly find $s \in \ZZ_q[t]$ and $K \in \NN$, 
where $s$ divides some power of the polynomial $R$ from Definition~\ref{defn:resultant}, such that 
$s \Phi$ is congruent modulo $p^{N_{\Phi}}$ to a matrix of polynomials of degree 
less than $K$. 
\end{thm}

\begin{proof}
On every residue disk $D \subset \mathbf{P}^1(\bar{\QQ}_q)$, we can 
apply either Theorem~\ref{thm:KedlayaTuitman} (when it applies) or 
Theorem~\ref{thm:Gerkmann} in order to find numbers $\theta_z \in \NN_0$ 
such that $\Phi$ is congruent modulo~$p^N$ to a matrix 
$\tilde{\Phi}\in M_{b \times b}(\QQ_q(t))$ for which $\ord_z(\tilde{\Phi}) \geq -\theta_z$ 
for all $z \in D$. 
By Theorem~\ref{thm:Gerkmann}, we can take $\theta_z=0$ for $z \neq \infty$ if 
$R(z) \neq 0$, so that $\theta_z$ vanishes for all but finitely many $z \in \mathbf{P}^1(\bar{\QQ}_q)$.
Moreover, by Theorem~\ref{thm:KedlayaTuitman}, we can also take $\theta_z=0$ 
whenever the matrix~$M$ does not have a pole in the residue disk at~$z$. 
Finally, we may assume that $\theta_z=\theta_{z'}$ if $z,z'$ are conjugates over~$\QQ_q$.
We now define
\begin{align*}
s &= \prod_{z \neq \infty, \ord_p(z) \geq 0} (t-z)^{\theta_z}
     \prod_{z \neq \infty, \ord_p(z) < 0}\left(\frac{t}{z}-1\right)^{\theta_z}, 
&K&= \left(\sum_{z} \theta_z \right) +1,
\end{align*}
which satisfy all the required conditions.
\end{proof}

We compute the matrix $\Phi$ to $t$-adic precision $K$ and $p$-adic
precision $N_{\Phi}$ using Algorithm~\ref{alg:expansion}. For any 
$\tau \in S(\bar{\FF}_q)$, we can now compute
\begin{equation*}
\Phi_{\tau} = 
  s(\hat{\tau})^{-1} \left( s \Phi \bmod{t^{K}} \right)|_{t=\hat{\tau}} 
  \bmod{p^{N_{\Phi}}}.
\end{equation*}
Since $\ord_p(s(\hat{\tau}))=0$, $\ord_p(\hat{\tau}) \geq 0$ and
$\ord_p(\Phi) \geq -\delta$, by Proposition~\ref{prop:matrixproductval} 
the matrix $\Phi_{\tau}$ will also be correct to precision $N_{\Phi}$ 
provided that $\hat{\tau}$ is computed to $p$-adic precision $N_{\Phi}+\delta$.

\subsection{Computing the zeta function}

Now we want to compute the zeta function of the fibre $X_{\tau}$ of
our family $X/S$ lying over some $\tau \in S(\FF_{\mathfrak{q}})$, 
where $\FF_{\mathfrak{q}}/\FF_q$ is a finite field extension.  
Recall from Theorem~\ref{thm:hypersurface} 
that the zeta function of $X_{\tau}$ is of the form
\begin{equation*}
Z(X_{\tau},T) = \frac{\chi(T)^{(-1)^n}}{(1 - T) (1 - \mathfrak{q}T) \dotsm (1 - \mathfrak{q}^{n-1}T)},
\end{equation*}
where $\chi(T) = \det \bigl( 1 - T \mathfrak{q}^{-1} \Frob_{\mathfrak{q}} | \Hrig^n(U_{\tau}) \bigr) \in \ZZ[T]$ 
denotes the reverse characteristic polynomial of the action 
of $\mathfrak{q}^{-1} \Frob_{\mathfrak{q}}$ 
on $\Hrig^n(U_{\tau})$.

We start by computing the matrix of the action 
of $\mathfrak{q}^{-1} \Frob_{\mathfrak{q}}$ on $\Hrig^n(U_{\tau})$. 
Let us still denote $a=\log_p(\mathfrak{q})$. Recall that $\Phi_{\tau}$ 
is the matrix of the action of $p^{-1} \Frob_p$ on~$\Hrig^{n}(U_{\tau})$ 
with respect to the basis $\cB$. As this action is $\sigma$-semilinear, 
we have that 
\begin{equation*}
\Phi_{\tau}^{(a)} = 
    \Phi_{\tau} \sigma(\Phi_{\tau}) \dotsm \sigma^{a-1}(\Phi_{\tau})
\end{equation*}
is the matrix of the action of $\mathfrak{q}^{-1} \Frob_{\mathfrak{q}}$ 
on $\Hrig^n(U_{\tau})$. 

We now analyse the loss of $p$-adic precision when computing 
the reverse characteristic polynomial 
$\chi(T)=1+\sum_{i=1}^b \chi_i T^i=\det\bigl( 1 - T \Phi_{\tau}^{(a)}\bigr)$.

\begin{thm} \label{thm:preccharpoly}
Let $N_{\Phi} \in \NN$ be such that $N_{\Phi} \geq \delta$, 
where $\delta$ is defined as in Definition~\ref{defn:delta}.
Moreover, suppose that $\tilde{\Phi}_{\tau}$ is an approximation to $\Phi_{\tau}$ satisfying
$\ord_p (\Phi_{\tau}-\tilde{\Phi}_{\tau}) \geq N_{\Phi}$.
Let us denote
\[
\tilde{\chi}(T) = 1 + \sum_{i=1}^b \tilde{\chi}_i T^i 
                = \det\left( 1 - T \tilde{\Phi}_{\tau}^{(a)}\right).
\]
Then for all $1 \leq i \leq b$, we have
\[
\ord_p \left(\chi_i - \tilde{\chi}_i \right) \geq N_{\Phi}-\delta.
\]
\end{thm}

\begin{proof} 
Recall from Theorem~\ref{thm:deltabound} that there exists a matrix 
$W_{\tau} \in M_{b \times b}(\QQ_q)$ satisfying 
$\ord_p(W_{\tau})+\ord_p(W_{\tau}^{-1}) \geq -\delta$ such that for 
$\Phi_{\tau}'=W_{\tau} \Phi_{\tau} \sigma(W)^{-1}$ 
we have \mbox{$\ord_p(\Phi_{\tau}') \geq 0$}.

Defining the matrix 
$\tilde{\Phi}_{\tau}'=W_{\tau} \tilde{\Phi}_{\tau} \sigma(W)^{-1}$, we find that 
$\ord_p(\Phi'_{\tau}-\tilde{\Phi}_{\tau}') \geq N-\delta$ and in particular
$\ord_p(\tilde{\Phi}_{\tau}') \geq 0$. Thus, we obtain
\[
\ord_p \left( \det\left(1 - T (\Phi'_{\tau})^{(a)}\right) 
            - \det\left(1 - T (\tilde{\Phi}'_{\tau})^{(a)}\right) \right) \geq N-\delta.
\] 
Note that $(\Phi'_{\tau})^{(a)}= W_{\tau} \Phi_{\tau}^{(a)} W_{\tau}^{-1}$
and $(\tilde{\Phi}'_{\tau})^{(a)}= W_{\tau} \tilde{\Phi}_{\tau}^{(a)} W_{\tau}^{-1}$, 
so that
\begin{align*}
\chi(T) &= \det\left(1 - T (\Phi'_{\tau})^{(a)}\right),
&\tilde{\chi}(T) &= \det\left(1 - T (\tilde{\Phi}'_{\tau})^{(a)}\right).
\end{align*}
This completes the proof.
\end{proof}

\begin{rem} \label{rem:workprecchi}
If we know $\Phi_{\tau}$ to $p$-adic precision $N_{\Phi}$, then
$\chi(T)$ is determined to precision $N_{\Phi}-\delta$.
However, we cannot compute $\chi(T)$ as in the proof of 
Theorem~\ref{thm:preccharpoly}, since we do not know the matrix 
$W_{\tau}$ explicitly. When computing with respect to our basis~$\cB$,
there will be loss of $p$-adic precision. The loss of precision 
in computing $\Phi_{\tau}^{(a)}$ from $\Phi_{\tau}$ is at most 
$(a-1)\delta$, and the loss of 
precision in computing $\chi(T)$ from $\Phi_{\tau}^{(a)}$ 
is at most $(b-1) \delta$, by Proposition~\ref{prop:productval} and 
Corollary~\ref{cor:delta}. Therefore, for $\chi(T)$ to be correct
to $p$-adic precision $N_{\Phi}-\delta$, it is sufficient to compute 
$\chi(T)$ from $\Phi_{\tau}$ using $p$-adic working precision 
\[
(N_{\Phi}-\delta)+\bigl((a-1)+(b-1)\bigr)\delta = N_{\Phi}+(a+b-3)\delta.
\] 
\end{rem}

When $p \geq n$, one can improve Theorem~\ref{thm:preccharpoly} by taking into 
account the Hodge numbers $h^{i,n-1-i}$ of $\HdR^{n}(\mathfrak{U}/\mathfrak{S})$.
Recall from Remark~\ref{rem:hnumbers} that $h^{i,n-1-i}=\card{B_i}$.

\begin{defn}
Define $\Gamma \colon [0,b] \rightarrow \RR$ to be the function whose graph 
is the convex polygon in the plane whose left-most point is the origin and 
which has slope~$i$ over the interval of the horizontal axis 
\[
\left[h^{0,n-1} + \dotsb + h^{i-1,n-i}, h^{0,n-1} + \dotsb + h^{i,n-i-1}\right].
\]
Note that $\Gamma(b)=b(n-1)/2$ since $h^{i,n-1-i}=h^{n-1-i,i}$ for all $0 \leq i \leq n-1$.
\end{defn}

\begin{thm} \label{thm:pgeqn}
We continue with the notation of Theorem~\ref{thm:preccharpoly}. 
Let \mbox{$N_{\Phi} \in \NN$} and suppose that $p \geq n$. If $\tilde{\Phi}_{\tau}$ 
is an approximation to $\Phi_{\tau}$ such that $\ord_p (\Phi_{\tau}-\tilde{\Phi}_{\tau}) \geq N_{\Phi}$ 
then, for all $1 \leq i \leq b$, we have
\[
\ord_p\left(\chi_i-\tilde{\chi}_i\right) \geq N_{\Phi} + a\Gamma(i-1).
\]
\end{thm}
 
\begin{proof} 
Note that in this case $\delta=0$ and hence that $\ord_p(\Phi_{\tau}) \geq 0$. 
By a result of Mazur~\citep[p.\ 665--666]{Mazur1972}, the map $p^{-1} \Frob_p$ 
sends $(H_i)_{\hat{\tau}} \cap \Lambda_{\tau,crys}$ into $p^i \Lambda_{\tau,crys}$, 
where $\Lambda_{\tau,crys}$ is defined as in the proof of 
Theorem~\ref{thm:deltabound} and $H_i$ as in Remark~\ref{rem:hnumbers}. 
This implies that the so called Hodge polygon of the 
$F$-crystal $(\Lambda_{\tau,crys},p^{-1} \Frob_p)$ lies above the graph of 
$\Gamma$. From this it follows that the Hodge polygon of the $F$-crystal 
$(\Lambda_{\tau,crys},\mathfrak{q}^{-1} \Frob_{\mathfrak{q}})$ lies
above the graph of $a\Gamma$. Hence we can find an invertible matrix 
$W_{\tau} \in M_{b \times b}(\ZZ_{\mathfrak{q}})$ such that for all 
$1 \leq i \leq b$, the sum of the valuations of any~$i$ different columns
of $W_{\tau} \Phi_{\tau}^{(a)} W_{\tau}^{-1}$ is at least $a\Gamma(i)$.
Note that 
\[ \ord_p(W_{\tau }\Phi_{\tau}^{(a)} W_{\tau}^{-1}-W_{\tau } \tilde{\Phi}_{\tau}^{(a)} W_{\tau}^{-1}) \geq N_{\Phi}.\]
The coefficients $\chi_i$ and $\tilde{\chi}_i$ are alternating sums 
of products of $i$ elements from different columns of 
$W_{\tau }\Phi_{\tau}^{(a)} W_{\tau}^{-1}$ 
and $W_{\tau } \tilde{\Phi}_{\tau}^{(a)} W_{\tau}^{-1}$, respectively.  
Therefore, the required bound follows from Theorem~\ref{prop:productval}. 
\end{proof}

\begin{rem}
For $a=1$, Theorem~\ref{thm:pgeqn} was obtained by Lauder in \citep[Proposition 9.4]{Lauder2006} and
the general idea of using the Hodge filtration to lower $p$-adic precision bounds had already been 
suggested before in \citep[Remark 1.6.4]{AbbottKedlayaRoe2006}. However, we have not been able
to find Theorem~\ref{thm:pgeqn} in the literature for $a \neq 1$. 
\end{rem}

Now we determine the $p$-adic precision~$N_{\chi_i}$ to which we need to compute~$\chi_i$ 
in order to recover the integer polynomial~$\chi(T)$ exactly.

\begin{thm} \label{thm:N0}
In order to recover the integer polynomial~$\chi(T)$ exactly, 
it suffices to compute $\chi_i$ to $p$-adic precision 
\begin{equation*}
N_{\chi_i} = \floorbig{\log_p \bigl( 2 \bigl( b/i \bigr) \mathfrak{q}^{i (n-1) / 2} \bigr)} + 1.
\end{equation*}
\end{thm}

\begin{proof}
By Theorem~\ref{thm:weildeligne} and the short exact 
sequence~\eqref{eqn:excision}, we have
\[
\chi(T)=\prod_{i=1}^b (1-\alpha_i T),
\]
where the $\alpha_i$ are algebraic integers of absolute 
value $\mathfrak{q}^{(n-1)/2}$ that are permuted under the 
map $\alpha \mapsto \mathfrak{q}^{n-1}/\alpha$. If we denote
$s_j = \sum_{i=1}^{b} \alpha_i^j$, then clearly
\[
|s_j| \leq b \mathfrak{q}^{j (n-1)/2}.
\]
for all $1 \leq j \leq b$. Moreover, if we write 
$\chi(T) = 1+\sum_{i=1}^{b} \chi_i T^i$, then by the Newton--Girard 
identies we have
\begin{equation} \label{eq:recursion}
s_j+j \chi_j = - \sum_{i=1}^{j-1} s_{j-i} \chi_i.
\end{equation}
This means that if we are given $\chi_1,\dotsc,\chi_{j-1}$, then 
we can limit $\chi_j$ to an explicit  disk in the complex plane of 
radius $(b/j) \mathfrak{q}^{j (n-1) / 2}$. Therefore, 
if we know each $\chi_i$ to $p$-adic precision $N_{\chi_i}$ satisfying
\[
p^{N_{\chi_i}} > \frac{2 b}{i} \mathfrak{q}^{i (n-1) / 2},
\] 
we can determine $\chi(T)$ exactly using the recursion~\eqref{eq:recursion}.
\end{proof}

\begin{rem}
This bound was first obtained by Kedlaya in~\citep[]{Kedlaya2007}, 
although it does not appear there in exactly this form.
\end{rem}

\begin{thm} \label{thm:precPhitau}
Let $N_{\chi_i}$ be defined as in Theorem~\ref{thm:N0}.  In order to 
compute $\chi(T)$ exactly, it is sufficient to compute~$\Phi_{\tau}$ 
with $p$-adic precision
\begin{align*}
N_{\Phi} &= \max_{1 \leq i \leq b} \left\{ N_{\chi_i} \right\} +\delta.
\intertext{If moreover $p \geq n$, then this can be improved to}
N_{\Phi} &= \max_{1 \leq i \leq b} \left\{ N_{\chi_i} -a\Gamma(i-1) \right\}.
\end{align*}
\end{thm}

\begin{proof}
This follows easily by combining Theorem~\ref{thm:N0} with 
Theorem~\ref{thm:preccharpoly} and Theorem~\ref{thm:pgeqn}, respectively.
\end{proof}

\begin{rem} \label{rem:epsilon}
If the sign $\epsilon = \pm 1$ is known for which 
$\det \bigl(\Phi_{\tau}^{(a)} \bigr) = \epsilon \mathfrak{q}^{b(n-1)/2}$, 
then this can be improved further. Since 
$\prod_{i=1}^b \alpha_i = \epsilon \mathfrak{q}^{b(n-1)/2}$, and 
the $\alpha_i$ are permuted under the map 
$\alpha \mapsto \mathfrak{q}^{n-1}/\alpha$, we have
\begin{equation*}
\chi_{b-i}=\epsilon (-1)^{b} \mathfrak{q}^{(n-1)(b/2-i)} \chi_b. 
\end{equation*}
So $\chi(T)$ is uniquely determined already by 
$\chi_1,\dotsc,\chi_{\lfloor b/2 \rfloor}$, and it is sufficient 
to take the maxima in Theorem~\ref{thm:precPhitau} running only over 
$1 \leq i \leq \floor{b/2}$.

It is well known that $\epsilon = 1$ when $n$ is even, but when $n$ 
is odd $\epsilon$ is usually not known. In practice it is then often 
still possible to use a smaller precision by computing $\epsilon$ first. 
Let $j$ be the smallest positive integer such that $\chi_{\ceil{b/2} - j} \neq 0$. 
To recover $\chi_0, \dotsc, \chi_{\floor{b/2}+j}$, it is sufficient to take 
the maxima in Theorem~\ref{thm:precPhitau} running only over 
$1 \leq i \leq \floor{b/2}+j$. This allows us to determine $\epsilon$ from 
the two coefficients $\chi_{\ceil{b/2}-j}$ and $\chi_{\floor{b/2}+j}$. 
\end{rem}

We now formalise the complete algorithm for computing $Z(X_{\tau},T)$ 
in Algorithm~\ref{alg:complete}.

\begin{algorithm} 
\caption{Compute $Z(X_{\tau},T)$.}
\label{alg:complete}
\begin{algorithmic}
\vspace{1mm}
\Require $P \in \ZZ_q[t][x_0,\dotsc,x_n]$ homogeneous of degree $d$ satisfying Assumption~\ref{assump:diag} and $\tau \in S(\mathbf{F}_{\mathfrak{q}})$.
\Ensure  The zeta function $Z(X_{\tau},T)$ of the fibre $X_{\tau}$ lying over $\tau$.
\Procedure{ZetaFunction}{$P,\tau$}
\State \begin{compactenum}[{\hspace{1em} } 1.] \vspace{-1.24em}
\item Determine $N_{\Phi}$ from Theorem~\ref{thm:precPhitau}.
\item $N_{\Phi}' \gets N_{\Phi}+(a+b-3) \delta$
\item Determine $K \in \NN$ and $s \in \ZZ_q[t]$ from Theorem \ref{thm:defD}.
\item Compute $\Phi \gets$ \textsc{FrobSeriesExpansion($N_{\Phi},K)$}.
\item Determine $\hat{\tau} \in \mathcal{S}(\ZZ_{\mathfrak{q}})$ to $p$-adic precision $N_{\Phi}+\delta$.
\item Compute $\Phi_{\tau} \gets s(\hat{\tau})^{-1} \bigl( a \Phi \bmod{t^{K}} \bigr)|_{t=\hat{\tau}}$ 
      to $p$-adic precision $N_{\Phi}$, using $p$-adic working precision $N_{\Phi}+\delta$.
\item Compute $\chi(T) \gets \det\bigl(1-T \bigl(\Phi_{\tau} \sigma(\Phi_{\tau}) \dotsm \sigma^{a-1}(\Phi_{\tau}) \bigr)  \bigr)$ to $p$-adic precision $N_{\Phi}-\delta$, 
      using $p$-adic working precision $N'_{\Phi}$.
\item Round $\chi(T)$ to $\ZZ[T]$ using the recursion~\eqref{eq:recursion}. 
\item Set $Z(X_{\tau},T) \gets \bigl( \chi(T)^{(-1)^n} \bigr)/\bigl((1 - T) (1 - \mathfrak{q}T) \dotsm (1 - \mathfrak{q}^{n-1}T)\bigr)$.
\item \Return $Z(X_{\tau},T)$
\end{compactenum}
\EndProcedure
\end{algorithmic}
\end{algorithm}

\begin{rem}
The output of Algorithm~\ref{alg:complete} only depends on $X_{\tau}$.  
In particular, we can also take a polynomial 
$\bar{P} \in \FF_q[t][x_0,\dotsc,x_n]$ as input and at the start of 
the algorithm take $P$ to be an arbitrary lift of $\bar{P}$. For the 
complexity analysis in the next section, we thus take the input size 
to be the size of $\bar{P}$.  However, from a practical point of view, 
it is convenient to keep the lift $P$ as the input, since:
\begin{enumerate} 
\item the matrices $M, \Phi_0, C, \Phi, \Phi_{\tau}$ do depend on $P$,
\item Assumption~\ref{assump:diag} needs to hold for $P$,
\item the runtime and memory usage of the algorithm depend on $P$.
\end{enumerate}
\end{rem}


\section{Complexity}

\label{sec:Complexity}

In this section we determine the complexity of Algorithm~\ref{alg:complete}.
We denote $a' = \log_p(q)$, noting that $a'$ divides $a$, 
and let $d_t$ denote the degree of $P$ in the variable~$t$. 

We use the $\SoftOh(-)$ notation that ignores logarithmic factors, 
i.e.\ $\SoftOh(f)$ denotes the class of functions 
that lie in $\BigOh(f \log^k(f))$ for some $k \in \NN$.

Let us first revisit some results concerning the complexity of the, 
sometimes basic, constituent operations.  We recall that two $k$-bit 
integers can be multiplied in $\SoftOh(k)$ bit operations, and that 
two degree-$k$ polynomials can be multiplied in $\SoftOh(k)$ ring 
operations.  We let $\omega$ denote an exponent for matrix 
multiplication, so that two $k \times k$ matrices can be multiplied 
in $\BigOh(k^{\omega})$ ring operations. An invertible $k \times k$ matrix 
can then be inverted in $\BigOh(k^{\omega})$ ring operations as well. 
Moreover, it is known~\citep{Williams2012} that one can take $\omega \leq 2.3729$.  
Finally, we point out that the characteristic polynomial of a matrix can be 
computed in $\BigOh(b^{\omega} \log(b))$ field operations using 
an algorithm of Keller-Gehrig~\citep{KellerGehrig1985}.

Next we consider specific $p$-adic operations, referring the reader to 
Hubrechts~\citep{Hubrechts2010} for further details.  First, images of elements 
of $\QQ_{\mathfrak{q}}$ under~$\sigma^i$, for $0 < i < a$, can be computed to 
$p$-adic precision~$N$ in time $\SoftOh(a \log^2(p) + a N \log(p))$.  Second, 
the Teichm\"uller lift of an element of $\FF_{\mathfrak{q}}$ can be computed 
to precision~$N$ in time $\SoftOh(a N \log^2(p))$.

We start by estimating the degrees of the numerator and denominator of the
connection matrix~$M$.

\begin{prop}
The degrees of $H$, $R$ from Proposition~\ref{thm:denom}, and
$G$, $r$ from Section~\ref{sec:DifferentialSystem} are all
$\BigOh(n(de)^n d_t) \subset \SoftOh((de)^n d_t)$.
\end{prop}

\begin{proof}
Note that $\Delta_k$ in Definition~\ref{defn:resultant} is a square matrix 
of degree~$d_t$, with
\[
{kd-1 \choose n} \leq \biggl( \frac{e(kd-1)}{n} \biggr)^n < (de)^n
\]
columns, where $e$ denotes the base for the natural logarithm. The result 
follows easily from this. Note that the degrees of the numerators and 
denominators of all intermediate results in Algorithm~\ref{alg:Connection} 
are also $\SoftOh((de)^n d_t)$.
\end{proof}

Next we estimate the precisions that we require.

\begin{prop}
All $p$-adic precisions that we use lie in $\SoftOh(a d^n \log(d_t))$ 
and $\deg(s), K$ are both $\SoftOh(apd^n (de)^n d_t)$.
\end{prop}

\begin{proof}
Note that $b \in \BigOh(d^n)$ and $\delta \in \BigOh(n)$.  
In Theorem~\ref{thm:precPhitau}, we have
\begin{align*}
\max_{1 \leq i \leq b} \{N_{\chi_i}\} \in \BigOh\bigl(a n b + \log(b) \bigr) 
            &\subset \SoftOh\bigl(a d^n \bigr),
&N_{\Phi}   &\in \SoftOh(a d^n).
\end{align*}
The precisions $N'_{\Phi}$ and $N_{\Phi} \pm \delta$ in 
Algorithm~\ref{alg:complete} are then $\SoftOh(ad^n)$ as well.
It follows from Theorem~\ref{thm:Gerkmann} that in
Theorem~\ref{thm:defD} we can take
\begin{align*}
\deg{s} &\leq \deg  \Bigl( \prod_{k=2}^n \det(\Delta_k) \Bigr) + (p(n+h(N_{\Phi}))-n) \deg(\det(\Delta_{n+1})), \\
K &\leq \deg(s)+ 1+\deg \Bigl(\prod_{k=2}^n \Delta_k^{-1}\Bigr) + (p(n+h(N_{\Phi}))-n) \deg(\Delta_{n+1}^{-1}) + ph(N_{\Phi}) d_t. 
\end{align*} 
Consequently, we obtain
\[
\deg(s), K \in \SoftOh(apd^n (de)^n d_t).
\]

The $p$-adic precisions $N_{\Phi_0}$, $N_C$, $N_M$, $N_{C^{-1}}$, $N'_{C}$, 
and $N'_{C^{-1}}$ in Theorem~\ref{thm:Ni} are in $\SoftOh(ad^n \log(d_t))$, 
noting that the logarithms that appear there are to base~$p$. Finally, the 
remaining $p$-adic precisions $N'_M$ and $N'_{\Phi}$ are also 
$\SoftOh(a d^n \log(d_t))$ by Remark~\ref{rem:precgm} and 
Corollary~\ref{cor:NPhi0prime}, respectively.
\end{proof}

We now analyse the computation of the connection matrix~$M$.
\begin{prop}
The computation of the connection matrix~$M$ using 
Algorithm~\ref{alg:Connection} requires
\begin{align*}
\mbox{time: }  & \SoftOh\bigl(a a' \log(p) (d^{n(\omega+2)} e^{n(\omega+1)}+ d^{5n}e^{3n} ) d_t\bigr), \\
\mbox{space: } & \SoftOh(a a' \log(p) d^{4n}e^{3n} d_t).
\end{align*}
\end{prop}

\begin{proof}
We first need to construct and invert the matrices~$\Delta_k$.  This is 
dominated by the inversion, which requires $\BigOh\bigl((de)^{n \omega}\bigr)$ 
operations in the ring $\QQ_q[t]$ to $p$-adic 
precision~$\SoftOh(a d^n \log(d_t))$. As there are $\BigOh(n)$ of 
these matrices, this takes time 
\[
\SoftOh\bigl((de)^{n \omega} (a' \log(p)) (a d^n \log(d_t)) ((de)^n d_t)\bigr) = 
    \SoftOh\bigl(a a' \log(p) d^{n(\omega+2)} e^{n(\omega+1)} d_t\bigr).
\]
Then we multiply each of the monomials in our basis with 
$-k (\partial P / \partial t)$ for some $1 \leq k \leq n$ and 
reduce the product to the basis by repeatedly using 
Algorithm~\ref{alg:Decompose}.  For each of these monomials this 
takes time 
\[
\SoftOh\bigl((de)^{2n} (a' \log(p)) (a d^n \log(d_t)) ((de)^n d_t)\bigr) = 
    \SoftOh\bigl(a a' \log(p) d^{4n}e^{3n} d_t\bigr),
\]
because of the quadratic complexity of the matrix-vector product.  
There are $b \in \BigOh(d^n)$ monomials in our basis and hence this 
takes time $\SoftOh(a a' \log(p) d^{5n}e^{3n} d_t)$.

During this computation we have to store $\BigOh(n)$ matrices 
and vectors of size 
\[
\SoftOh\bigl((de)^{2n} (a' \log(p)) (a d^n \log(d_t)) ((de)^n d_t)\bigr),
\] 
which completes the proof.
\end{proof}

Next we consider the computation of the matrix $\Phi_0$.

\begin{prop} \label{prop:complexityPhi0}
The computation of the matrix~$\Phi_0$ with Algorithm~\ref{alg:Diagfrob} requires
\begin{align*}
\mbox{time: }  & \SoftOh\bigl(a^3 p d^{4n} \log^3(d_t)\bigr), \\
\mbox{space: } & \SoftOh\bigl(a \log(p) d^{2n} \log(d_t)\bigr).
\end{align*}
\end{prop}

\begin{proof}
In Proposition~\ref{prop:MR}, we have $\mathcal{R} \in \SoftOh(N'_{\Phi_0})$.
In Definition~\ref{defn:alpha}, the sums
\begin{equation*} \label{eqn:ifactor}
\left(\frac{u_i+1}{d} \right)_r \sum_{j=0}^{r} \frac{\bigl(p a_i^{p-1}\bigr)^{r-j}}{(m-pj)!j!}
\end{equation*}
consist of $\BigOh(\mathcal{R})$ terms. Note that each term can be computed from the previous one in $\BigOh(p)$ operations 
in $\ZZ_p$. For each $r$, the sum can therefore be computed in time $\BigOh(p \mathcal{R} N'_{\Phi_0})$. We have to consider
$\BigOh(\mathcal{R})$ values of $r$, so each $\alpha_{u,v}$ can be computed in time 
\[
\SoftOh\bigl(np \mathcal{R}^2 N'_{\Phi_0}\bigr) \subset \SoftOh\bigl(a^3 p d^{3n} \log^3(d_t)\bigr).
\]
We need to compute $b \in \BigOh(d^n)$ of these $\alpha_{u,v}$, so the computation of $\Phi_0$
takes time $\SoftOh(a^3 p d^{4n} \log^3(d_t))$. 

The required space is dominated by the size of the output, which is
\[
\SoftOh\bigl(b \log(p) N'_{\Phi_0}\bigr) \subset \SoftOh\bigl(a \log(p) d^{2n} \log(d_t)\bigr). \qedhere
\]
\end{proof}

\begin{rem}
Note that the time complexity in Proposition~\ref{prop:complexityPhi0} 
is quasilinear in~$p$, while in the work of Lauder~\citep{Lauder2004a} 
it is quasiquadratic.  The main reason for this is that Lauder computed 
in the totally ramified extension~$\QQ_p(\pi)$ with $\pi^{p-1}=-p$, 
where a single multiplication already takes time quasilinear in~$p$.  It will turn out 
that this crucial improvement also decreases the overall time complexity 
of the deformation method from being quasiquadratic to quasilinear in~$p$.

We should mention that there is a small downside to our approach.  The time 
complexity in Proposition~\ref{prop:complexityPhi0} is quasicubic in~$a$, 
while by using fast exponentials over $\QQ_p(\pi)[[z]]$, this can be 
decreased to being quasiquadratic in~$a$, which again has an effect on the 
entire deformation method.  We address this in the following proposition.
\end{rem}

\begin{prop} \label{prop:compPhi0}
Alternatively, the matrix~$\Phi_0$ can be computed in 
\begin{align*}
\mbox{time: }  &\SoftOh\bigl(a^2 p d^{2n} (p + d^n) \log^2(d_t) \bigr), \\
\mbox{space: } &\SoftOh\bigl(a^2 p^2 d^{2n} \log^2(d_t)\bigr).
\end{align*}
\end{prop}

\begin{proof}
In Proposition~\ref{prop:MR}, we have $\mathcal{M} \in \SoftOh(pN'_{\Phi_0})$.
Let $\QQ_p(\pi)$ denote the totally ramified extension of $\QQ_p$ with $\pi^{p-1}=-p$.
Recall from the proof of Proposition~\ref{prop:coefbound} that in Definition~\ref{defn:alpha},
up to a factor $(-1)^r (\pi a_i)^{m-(p-1)r}$, the sum
\begin{align*}
\sum_{j=0}^{r} \frac{(p a_i^{p-1})^{r-j}}{(m-pj)!j!}
\end{align*}
is equal to the coefficient $\lambda_{m}$ of $x^m$ in the power series expansion of $\exp(\pi a_i(x-x^p))$. However,
the power series expansion of $\exp(\pi a_i(x-x^p))$ modulo $x^{\mathcal{M}+1}$ can be computed in 
$\SoftOh(\mathcal{M})$ operations in $\QQ_p(\pi)$ following Brent~\citep{Brent1976}. Since
a single operation in $\QQ_p(\pi)$ takes time $\SoftOh(pN'_{\Phi})$, precomputing these
power series expansions for all $0 \leq i \leq n$ takes time 
\[
\SoftOh\bigl(n p^2 (N'_{\Phi_0})^2\bigr) \subset \SoftOh\bigl(a^2 p^2 d^{2n} \log^2(d_t)\bigr).
\]
Each $\alpha_{u,v}$ can now be computed as in Proposition~\ref{prop:complexityPhi0} in time 
\[
\SoftOh\bigl(p(N'_{\Phi_0})^2 \bigr) \subset \SoftOh\bigl(a^2 p d^{2n} \log^2(d_t)\bigr),
\]
so computing all $b \in \BigOh(d^n)$ of these $\alpha_{u,v}$ takes time 
$\SoftOh(a^2 p^2 d^{3n} \log^2(d_t))$. Hence $\Phi_0$
can be computed in time $\SoftOh(a^2 p d^{2n} (p + d^n) \log^2(d_t))$.

The space requirement is dominated by that of the power series expansions of $\exp(\pi a_i(x-x^p))$ modulo $x^{\mathcal{M}+1}$ for $0 \leq i \leq n$
and is therefore given by
\[
\SoftOh\bigl(n \mathcal{M} p N'_{\Phi_0}\bigr) \subset \SoftOh\bigl(a^2 p^2 d^{2n} \log^2(d_t)\bigr). \qedhere
\]
\end{proof}

We now consider the computation of the power series expansion of the 
matrix~$\Phi$.

\begin{prop}
Assume that the matrices $M = G/r$ and $\Phi_0$ have been computed already.
The subsequent computation of the power series expansion of the matrix~$\Phi$ 
in Algorithm~\ref{alg:expansion} then requires
\begin{align*}
\mbox{time: }  &\SoftOh(a^2 a' p d^{n(\omega+4)}e^{2n} d_t^2), \\
\mbox{space: } &\SoftOh(a^2 a' p d^{5n} e^n d_t).
\end{align*}
\end{prop}

\begin{proof}
The computation of the power series expansion of~$\Phi$ comprises 
three steps, namely the computation of the matrices $C$, $\sigma(C)^{-1}$ 
and the matrix product $C \Phi_0 \sigma(C)^{-1}$.

As each of the $K$ steps in the computation of $C$ is dominated by the 
computation of $\SoftOh((de)^n d_t)$ matrix products, the matrix~$C$ can be 
computed in time 
\begin{equation*}
\SoftOh\bigl(K ((de)^n d_t) b^{\omega} (a' \log (p)) (a d^n \log (d_t))\bigr) 
    \subset \SoftOh\bigl(a^2 a' p d^{n(\omega + 4)} e^{2n} d_t^2 \bigr).
\end{equation*}
Similarly, the matrix $C^{-1}$ can be computed in time 
\begin{equation*}
\SoftOh\bigl( (K/p) ((de)^n d_t) b^{\omega} (a' \log (p)) (ad^n \log (d_t)) \bigr)
    \subset \SoftOh\bigl( a^2 a' \log(p) d^{n(\omega+4)} e^{2n} d_t^2 \bigr). 
\end{equation*}
Moreover, applying $\sigma$ to the matrix $C^{-1}$ takes time 
\begin{equation*}
\SoftOh\bigl( (K/p) b^2 \bigl(a' \log^2(p) + a' (a d^n \log(d_t)) \log(p) \bigl) \bigr)
\subset \SoftOh\bigl( a^2 a' \log^2(p) d^{5n} e^n d_t \bigr).
\end{equation*}
Finally, the matrix product 
$C \Phi_0 \sigma(C)^{-1}$ can be computed in time 
\begin{equation*}
\SoftOh\bigl( b^{\omega} K (a' \log(p)) (a d^n \log(d_t)) \bigr) 
    \subset \SoftOh\bigl( a^2 a' p d^{n(\omega + 3)} e^n d_t \bigr).
\end{equation*}
The result on the time complexity now follows.

The space requirement is dominated by the matrix~$C$, which has size
\begin{equation*}
\SoftOh\bigl(b^2 K (a' \log(p)) (a d^n \log(d_t))\bigr) 
    \subset \SoftOh(a^2 a' p d^{5n} e^n d_t). \qedhere
\end{equation*}
\end{proof}

We now move on to the computation of the matrix~$\Phi_{\tau}$.

\begin{prop}
The computation of the matrix~$\Phi_{\tau}$ from the matrix~$\Phi$ 
and $\tau \in S(\FF_{\mathfrak{q}})$ requires
\begin{align*}
\mbox{time: }  &\SoftOh\bigl(a^2 a' p d^{5n} e^n d_t \bigr), \\ 
\mbox{space: } &\SoftOh\bigl(a^2 a' p d^{5n} e^n d_t \bigr).
\end{align*}
\end{prop}

\begin{proof}
We first recall that the Teichm\"uller lift 
$\hat{\tau} \in \mathcal{S}(\ZZ_{\mathfrak{q}})$ 
can be computed to $p$-adic precision $N_{\Phi}+\delta \in \SoftOh(a d^n)$ 
in time $\SoftOh(a^2 d^n \log^2(p))$.  Next, we observe that the 
scalar-matrix product $s \Phi \bmod t^K$ over~$\QQ_q[t]$ 
requires time 
\begin{equation*}
\SoftOh\bigl( b^2 K a' (a d^n) \log(p) \bigr) 
    \subset \SoftOh\bigl( a^2 a' p d^{5n} e^n d_t \bigr).
\end{equation*}
Finally, we consider the substitution of $\hat{\tau}$ 
into the $b^2$ entries of the matrix $s \Phi \bmod t^K$. 
Each of these can be thought of as a modular composition of polynomials 
over~$\QQ_{q}$, where the modulus~$m(t)$ is an irreducible polynomial defining 
the extension $\QQ_{\mathfrak{q}} / \QQ_{q}$ as a quotient of $\QQ_q[t]$, 
which is of degree~$a/a'$. However, care has to be taken to include the 
additional reduction modulo~$m(t)$ of the polynomials of degree less than~$K$, 
so that polynomials involved in the modular composition have degree less 
than~$a/a'$. 

Thus, the substitutions require time 
\begin{equation*}
\SoftOh\Bigl( b^2 \bigl(K (a' \log(p)) (a d^n) + (a/a') (a' \log(p)) (a d^n) \bigr) \Bigr)
    \subset \SoftOh(a^2 a' p d^{5n} e^n d_t).
\end{equation*}
Clearly, evaluating and inverting $s(\hat{\tau})$ and performing the 
scalar multiplication can be ignored, and the result on the time complexity now follows.

The space requirement is dominated by the matrix $s \Phi \bmod t^K$, 
which has size 
\begin{equation*}
\SoftOh(b^2 K (a' \log(p)) (a d^n)) \subset \SoftOh(a^2 a' p d^{5n} e^n d_t). 
\qedhere
\end{equation*}
\end{proof}

Finally, we consider the computation of the polynomial $\chi(T)$.

\begin{prop}
The computation of $\chi(T)$ from $\Phi_{\tau}$ requires
\begin{align*}
\mbox{time: }  & \SoftOh(a^2  \log^2(p) d^{n(\omega+1)}), \\
\mbox{space: } & \SoftOh(a^2 \log(p) d^{3n}).
\end{align*}
\end{prop}

\begin{proof}
In order to compute $\Phi_{\tau}^{(a)}$ using fast exponentiation for 
semilinear maps as in \citep[Lemma 32]{LauderWan2008}, we first need to apply powers of $\sigma$ to $\BigOh(b^2 \log(a))$ 
elements of $\QQ_{\mathfrak{q}}$ and then multiply $\BigOh(\log (a))$ matrices 
of size $b \in \BigOh(d^n)$. This can be done in time 
\begin{gather*}
\SoftOh\Bigl( b^2 \log(a) \bigl( a \log^2(p) + (a \log(p))(a d^n)  \bigr) 
    + \log(a) b^{\omega} (a \log(p))(a d^n)  \Bigr) 
\subset \SoftOh(a^2 \log^2(p) d^{n(\omega+1)}).
\end{gather*}
Next, we compute the reverse characteristic polynomial of matrix 
$\Phi_{\tau}^{(a)} \in M_{b \times b}(\QQ_{\mathfrak{q}})$, which 
can be accomplished in~$\BigOh(b^{\omega} \log(b))$ field operations. 
This amounts to a time complexity of 
\begin{equation*}
\SoftOh\bigl(b^{\omega} (a \log(p))(a d^n) \bigr)
    \subset \SoftOh\bigl( a^2 \log(p) d^{n(\omega+1)} \bigr).
\end{equation*}
Rounding this polynomial to $\ZZ[T]$ can be ignored.

We need to store $\BigOh(\log(a))$ matrices of size~$b$ with entries 
in~$\QQ_{\mathfrak{q}}$.  This requires space 
$\SoftOh(b^2 (a \log(p)) (a d^n)) \subset \SoftOh(a^2 \log(p) d^{3n})$.
\end{proof}

We can now state the total time and space requirements of 
Algorithm~\ref{alg:complete}. Recall that $P \in \ZZ_q[t][x_0,\dotsc,x_n]$ 
denotes a homogeneous polynomial
of degree $d$ satisfying Assumption~\ref{assump:diag} 
and that $\tau \in S(\mathbf{F}_{\mathfrak{q}})$, where 
$\FF_{\mathfrak{q}}/\FF_q$ denotes a finite field extension
of characteristic $p$. 
Moreover, we write $a=\log_p(\mathfrak{q})$, $a'=\log_p(q)$ and let
$d_t$ be the degree of $P$ in the variable $t$. Finally,
we let $e$ denote the base of the natural
logarithm and $\omega$ an exponent for matrix multiplication.

\begin{thm} \label{thm:totalcomplexity}
The computation of $Z(X_{\tau},T)$ using Algorithm~\ref{alg:complete} requires
\begin{align*}
\mbox{time: }  &\SoftOh\Bigl( 
                a^3 p d^{4n} + 
                a^2 a' p d^{n(\omega+4)} e^{2n} d_t^2+
                aa' \bigl(d^{n(\omega+2)} e^{n(\omega+1)}+ d^{5n} e^{3n}\bigr) d_t 
                \Bigr), \\
\mbox{space: } &\SoftOh\bigl( a^2 a' p d^{5n} e^n d_t + a a' d^{4n} e^{3n} d_t \bigr).
\end{align*}
Alternatively, computing the matrix $\Phi_0$ as in the proof of Proposition~\ref{prop:compPhi0}, 
the computation of $Z(X_{\tau},T)$ requires
\begin{align*}
\mbox{time: }  &\SoftOh\Bigl( 
                a^2 p^2 d^{2n} + 
                a^2 a' p d^{n(\omega+4)} e^{2n} d_t^2+
                aa' \bigl(d^{n(\omega+2)} e^{n(\omega+1)}+ d^{5n} e^{3n}\bigr) d_t 
                \Bigr), \\
\mbox{space: } &\SoftOh\bigl( 
                a^2 p^2 d^{2n} + 
                a^2 a' p d^{5n} e^n d_t + 
                a a' d^{4n} e^{3n} d_t \bigr).
\end{align*}
\end{thm}

\begin{proof}
This follows by adding all complexities from the previous propositions and 
leaving out terms that are dominated by other terms or powers of
logarithms of other terms.
\end{proof}

\begin{rem} In \citep{Lauder2004a}, Lauder took $d_t=1$ and showed that his
algorithm requires
\begin{align*}
\mbox{time: }  &\SoftOh\bigl(a^3 p^2 (d^{n(\omega+5)} e^{3n} + d^{6n} e^{5n}) \bigr), \\ 
\mbox{space: } &\SoftOh\bigl(a^3 p^2 d^{6n} e^{4n} \bigr).
\end{align*}
His main goal was to show that these complexities are $(pad^n)^{\BigOh(1)}$, 
which was not the case for previously known algorithms such as 
\citep{AbbottKedlayaRoe2006, LauderWan2008}. 
Our Theorem~\ref{thm:totalcomplexity} improves Lauder's complexity bounds 
by lowering the constants implicit in the exponent~$\BigOh(1)$.  To our 
knowledge, the complexity bounds presented in Theorem~\ref{thm:totalcomplexity} 
are therefore the best ones known.  Note that since a natural measure for the 
input size is $\log(p) a d^n$, these bounds are only polynomial in the 
input size provided $p$ is fixed, which is something that all $p$-adic 
point counting algorithms tend to have in common.
\end{rem}


\section{Examples}
\label{sec:Examples}

In this section we apply Algorithm~\ref{alg:complete} to some examples 
using our implementation\footnotemark.  This implementation 
is restricted to the case $q = p$, i.e.\ to families of hypersurfaces 
defined over a prime field.  In order to provide some context for the 
runtimes presented below, note that all computations were carried out 
on a machine with two Intel Core i7-3540M processors running at 3GHz and 
with 8GB of RAM, but only used a single processor.  Moreover, timings were obtained using the C function 
{\tt{clock()}} and are stated in minutes~(m) or seconds~(s). 

\footnotetext{This implementation is available at 
\url{https://github.com/SPancratz/deformation}.}

\subsection{Quintic curve}

We consider the family of genus six curves over $\ZZ$ given by the polynomial 
\begin{equation*}
P=x_0^5 + x_1^5 + x_2^5 + t x_0 x_1 x_2^3,
\end{equation*}
which Gerkmann~\citep[\S 7.4]{Gerkmann2007} considers as an element of 
$\ZZ_p[t][x_0,x_1,x_2]$ for $p=2,3,7$.  Note that these are, in fact, 
three examples.  The connection matrix $M \in M_{12 \times 12}(\QQ(t))$ 
only has to be computed once, and turns out to have denominator $r=27t^5+3125$.
The set of exponents at each of the zeros of $r$ is $\{-1,0\}$, and 
after changing basis by
\[ 
W=\mbox{diag}(t^{-1},t,t^{-2},t^{-1},1,1,1,1,1,t^2,t^2,1)
\] 
as in Remark~\ref{rem:basischange}, the set of exponents at~$\infty$ is 
$\{-2/3,2/3,1,4/3,7/3,8/3,13/3,14/3\}$.  Note that by
Remark~\ref{rem:epsilon}, we only need to determine the bottom 
half of the coefficients of the polynomial $\chi(T)$ directly.

\subsubsection{Precisions}

\noindent
\textit{Prime $p=2$.} 
We take $\log_p(\mathfrak{q})=50$ and let $\tau \in \FF_{\mathfrak{q}}$ be 
a zero of the Conway polynomial of $\FF_{\mathfrak{q}}/\FF_p$, i.e. the 
standard irreducible polynomial used to represent this extension.  We first 
compute $\delta=0$ and $N_{\Phi}=N_{\Phi}'=153$.  As the roots of~$r$ are 
$2$-adic integers and distinct  modulo~$2$, we can apply 
Theorem~\ref{thm:KedlayaTuitman} everywhere. We find $\theta_z=320$ at all 
zeros~$z$ of $r$ and $\theta_{\infty}=12$, so that we can take $s=r^{320}$ 
and $K=1613$. We now compute the remaining $p$-adic precisions
$N_{\Phi_0}=174$, $N_C=163$, $N_{C^{-1}}=164$, $N_M=187$, 
$N_C'=184$, $N_{C^{-1}}'=183$, $N_{\Phi_0}'=176$, and $N_M'=188$. \\

\noindent 
\textit{Prime $p=3$.}
We take $\log_p(\mathfrak{q})=40$ and let $\tau \in \FF_{\mathfrak{q}}$ be 
a zero of the Conway polynomial of $\FF_{\mathfrak{q}}/\FF_p$.  We first 
compute $\delta=0$ and $N_{\Phi}=N_{\Phi}'=122$. As the roots of~$r$ are all 
contained in the residue disk at~$\infty$, we cannot apply 
Theorem~\ref{thm:KedlayaTuitman} and have to apply Theorem~\ref{thm:Gerkmann} 
instead. At each of the zeros $z$ of~$r$ we find $\mu_z=0$, $\nu_z=-1$, $\theta_z=397$ 
and at $\infty$ we find $\mu_{\infty}=-3$, $\nu_{\infty}=-2$, $\theta_{\infty}=793$, so that 
we can take $s=r^{397}$ and $K=2779$. We now compute the remaining $p$-adic 
precisions
$N_{\Phi_0}=137$, $N_C=129$, $N_{C^{-1}}=130$, $N_M=147$, 
$N_C'=146$, $N_{C^{-1}}'=143$, $N_{\Phi_0}'=139$, and $N_M'=147$. \\

\noindent
\textit{Prime $p=7$.}
We take $\log_p(\mathfrak{q})=10$, and let $\tau \in \FF_{\mathfrak{q}}$ be 
a zero of the Conway polynomial of $\FF_{\mathfrak{q}}/\FF_p$. We first 
compute $\delta=0$ and $N_{\Phi}=N_{\Phi}'=31$.  As the zeros of $r$ are 
$7$-adic integers and distinct modulo~$7$, we can apply 
Theorem~\ref{thm:KedlayaTuitman} everywhere.  We find $\theta_z=224$ at all 
zeros~$z$ of $r$ and $\theta_{\infty}=25$, so that we can take $s=r^{224}$ and
$K=1146$. We now compute the remaining $p$-adic precisions
$N_{\Phi_0}=38$, $N_C=34$, $N_{C^{-1}}=37$, $N_M=44$, 
$N_C'=43$, $N_{C^{-1}}'=40$, $N_{\Phi_0}'=40$, and $N_M'=44$.

\subsubsection{Timings}

We now compare the performance of our FLINT implementation with the 
timings reported by Gerkmann. To save space, we do not include the 
polynomials $\chi(T)$ that we obtained from these computations.

\begin{center}
\begin{tabular}{l l l l l l l} \toprule
                & \multicolumn{2}{l}{$p = 2$} 
                & \multicolumn{2}{l}{$p = 3$} 
                & \multicolumn{2}{l}{$p = 7$} \\ \midrule
Computation     & P--T & G     & P--T  & G       & P--T & G     \\ \midrule
$M$             & 0.00s&       & 0.00s &         & 0.00s&       \\
$\Phi_0$        & 0.03s& 2.65m & 0.03s & 5.86m   & 0.01s& 1.31m \\
$\Phi$          & 0.67s& 1.33m & 1.28s & 1.38m   & 0.29s& 0.89m \\
$Z(X_{\tau},T)$ & 9.20s& 3.96m & 6.42s & 101.40s & 0.15s& 1.09m \\
Total           & 9.90s& 7.94m & 7.73s & 8.93m   & 0.45s& 3.29m \\ \bottomrule
\end{tabular} 
\end{center}

\bigskip

Our timings in these examples are a factor of $50-500$ lower than the ones 
provided by Gerkmann.

\subsection{Quartic surface}

We consider the family of quartic K3 surfaces over $\ZZ_3$ given by
\begin{equation*}
P=x_0^4 + x_1^4 + x_2^4 + x_3^4 + t x_0 x_1 x_2 x_3,
\end{equation*}
which Gerkmann considers in~\citep[\S 7.5]{Gerkmann2007}. The connection matrix 
$M \in M_{21 \times 21}(\QQ(t))$ turns out to have denominator 
$r=t^4-256$. The set of exponents at each of the zeros of $r$ is $\{-3/2,-1/2,0\}$, 
and after changing basis by
\[
W=\mbox{diag}(t^{-2},1,1,1,1,1,1,1,1,1,t^{-1},1,1,1,1,1,1,1,1,1,1)
\] 
as in Remark~\ref{rem:basischange}, the set of exponents  at $\infty$ is 
$\{1,2,3\}$. 

\subsubsection{Precisions}

We take $a=\log_3(\mathfrak{q})=20$, let $\alpha \in \FF_{\mathfrak{q}}$
be a zero of the Conway polynomial of $\FF_{\mathfrak{q}}/\FF_{3}$, and take
$\tau=\alpha^{2345}$.  We first compute $\delta=0$ and $N_{\Phi}=N_{\Phi}'=43$.  
Since the zeros of $r$ are $3$-adic integers and different modulo $3$, we can apply 
Theorem~\ref{thm:KedlayaTuitman} everywhere. We find $\theta_z=148$ at all zeros 
of $r$ and $\theta_{\infty}=6$, so that we can take $s=r^{148}$
and $K=599$. We now compute the remaining $p$-adic 
precisions $N_{\Phi_0}=65$, $N_C=53$, $N_{C^{-1}}=55$, $N_M=74$, 
$N_C'=73$, $N_{C^{-1}}'=68$, $N_{\Phi_0}'=68$ and $N_M'=75$.

\subsubsection{Timings}

\begin{center}
\begin{tabular}{l l l} \toprule
Computation     & P--T & G      \\ \midrule
$M$             & 0.00s& 7.00s      \\
$\Phi_0$        & 0.01s& 45.26m \\
$\Phi$          & 0.22s& 19.90m \\
$Z(X_{\tau},T)$ & 0.76s& 18.66m \\
Total           & 0.99s& 83.82m \\ \bottomrule
\end{tabular}
\end{center}

\bigskip

As the final result of the computation we find that 
$\mathfrak{q} \chi(T/\mathfrak{q})$ is equal to 
\begin{multline*}
- 3486784401T^{21} - 39675197243T^{20} - 191506614866T^{19} - 482588946510T^{18} \\
- 552821487569T^{17} + 243001138765T^{16} + 1641410078472T^{15} + 1793016627512T^{14} \\
- 410199003010T^{13} - 2617001208822T^{12} - 1586643774924T^{11} + 1586643774924T^{10} \\
+ 2617001208822T^9 + 410199003010T^8 - 1793016627512T^7 - 1641410078472T^6 \\
- 243001138765T^5+ 552821487569T^4 + 482588946510T^3 + 191506614866T^2 \\
+ 39675197243T + 3486784401.
\end{multline*}

Our timings in this example are about $5,000$~times lower than the ones 
presented by Gerkmann. 

Additionally, we should mention that Gerkmann does not include timings for part 
of his computations.  For example, he omits the time that is required for the 
computation of the $p$-adic precision parameter that he calls $\delta$, which 
involves solving $b^2$~differential equations similar to the one for the matrix~$C$. 
Moreover, it appears that there are minor errors in the precision analysis, e.g.\ 
the final $p$-adic precision is not sufficient to recover the exact zeta function. 
Finally, the polynomial~$\chi(T)$ provided by Gerkmann for this example does 
not satisfy Theorem~\ref{thm:weildeligne}, so cannot be correct.

The remaining examples could not have been computed with previous implementations 
of the deformation method and, to our knowledge, neither with any other point counting 
method.  For example, Lauder noted that he could not compute the connection matrix 
for a family of quintic curves or quartic surfaces given by a polynomial with more 
than a \emph{few} nonzero terms.  Indeed, this was the main motivation for the work 
in the PhD~thesis of the first author and the present paper.  We are now able to 
compute the zeta function of e.g.\ quintic curves and quartic surfaces over small 
finite fields given by a polynomial with \emph{all} of its coefficients nonzero.

\subsection{Generic quintic curve}

We consider the family of generic quintic curves over $\ZZ_{11}$ given by 
\begin{equation*}
\begin{split}
P = x_0^5 + x_1^5 + x_2^5 
+ t \bigl( & 3x_0^4 x_1-x_0^4 x_2+2 x_0 x_1^4+x_1^4 x_2+4 x_0 x_2^4+5 x_1 x_2^4 + x_0^3 x_1^2 \\
           & + x_0^3 x_2^2 + x_0^2 x_1^3+ x_1^3 x_2^2+x_0^2 x_2^3+x_1^2 x_2^3 + x_0^3 x_1 x_2 + x_0 x_1^3 x_2 \\
           & + x_0 x_1 x_2^3 + x_0^2 x_1^2 x_3 + x_0^2 x_1 x_2^2 + x_0 x_1^2 x_2^2 \bigr).
\end{split}
\end{equation*}
The connection matrix $M \in M_{12 \times 12}(\QQ(t))$ has denominator 
$r = r_1 r_2$, with polynomials $r_1,r_2 \in \ZZ[t]$ that are irreducible 
of degree $30$ and $48$, respectively. The set of exponents is $\{ 0,1 \}$ 
at the zeros of $r_1$ and $\{ -1,0 \}$ at the zeros of~$r_2$. Moreover, 
the matrix $M$ has a simple pole at $\infty$ and the set of exponents 
is $\{1,2\}$ there.

\subsubsection{Precisions}

We take $a=\log_{11}(\mathfrak{q})=10$ and let $\tau$ be a zero of the Conway polynomial of 
$\FF_{\mathfrak{q}}/\FF_{11}$. 
We first compute $\delta=0$ and $N_{\Phi}=N_{\Phi}'=31$.
Since the zeros of $r$ are $11$-adic integers and different modulo $11$, we can apply 
Theorem~\ref{thm:KedlayaTuitman} everywhere. We find $\theta_z=352$ at the zeros of $r_2$
and $\theta_{\infty}=9$ at $\infty$. At the residue disk of a zero $z$ of $r_1$, noting that 
$R$ has no other zeros there, we get a better bound by applying 
Theorem~\ref{thm:Gerkmann} with $\mu_z=-1$, $\nu_z=0$ and find $\theta_z=1$, so that we can 
take $s=r_1 r_2^{352}$ and $K=16936$. We now compute the remaining $p$-adic precisions 
$N_{\Phi_0}=40$, $N_C=35$, $N_{C^{-1}}=36$, $N_M=47$, $N_C'=46$, $N_{C^{-1}}'=43$, 
$N_{\Phi_0}'=42$ and $N_M'=47$.

\subsubsection{Timings}

\begin{center}
\begin{tabular}{l l} \toprule
Computation     & P--T \\ \midrule
$M$             & 4.43s     \\
$\Phi_0$        & 0.04s     \\
$\Phi$          & 3.76m     \\
$Z(X_{\tau},T)$ & 9.48s     \\
Total           & 3.99m     \\ \bottomrule
\end{tabular}
\end{center}

\bigskip

As the final result of the computation we find that~$\chi(T)$ is equal to 
\begin{equation*}
\begin{split}
& 304481639541418099574449295360278774639038415066698088621947601 T^{12} \\
& \quad + 3777543732986291528931322507772938448980494046897871937792 T^{11} \\
& \quad + 23639674223084796290417361507756397439403378558130350 T^{10} \\ 
& \quad + 54141350391870148138663709375646947695242620108 T^9 \\
& \quad - 363231942297281636316475334779570949613459 T^8 \\
& \quad - 4138991673785569248268236480720472276 T^7 \\
& \quad - 28015243113507339254470240817992 T^6 - 159576046483273177468242676 T^5 \\ 
& \quad - 539921137173243550659 T^4 + 3102762464729708 T^3 \\
& \quad + 52231690350 T^2 + 321792 T + 1.
\end{split}
\end{equation*}

\subsection{Generic quartic surface}

We consider the family of generic quartic K3 surfaces over $\ZZ_{7}$ given by 
\begin{equation*}
\begin{split}
P=x_0^4 + x_1^4 + x_2^4 + x_3^4 
+ t \bigl( & -3 x_0^3 x_1 + 2 x_0^3 x_2 - 2 x_0 x_1 x_2 x_3 + x_0^3 x_3 - x_0 x_1^3 - 3 x_1^3 x_2 \\ 
           & + x_2^3 x_3 + 2 x_1^3 x_3 + x_0 x_2^3 - 2 x_1 x_2^3 - x_0 x_3^3 + x_1 x_3^3 \\
           & + 3 x_2 x_3^3 + x_0^2 x_1^2 + 3 x_0^2 x_2^2 + x_0^2 x_3^2 + 2 x_1^2 x_2^2 - 2 x_1^2 x_3^2 \\
           & + x_2^2 x_3^2 + 2 x_0^2 x_1 x_2 + x_0^2 x_1 x_3 + 3 x_0^2 x_2 x_3 - x_0 x_1^2 x_2 \\
           & + 2 x_0 x_1^2 x_3 + 3 x_1^2 x_2 x_3 - x_0 x_1 x_2^2 + 3 x_0 x_2^2 x_3 + x_1 x_2^2 x_3 \\
           & + 2 x_0 x_1 x_3^2 + 2 x_0 x_2 x_3^2 + 2 x_1 x_2 x_3^2 \bigr).
\end{split}
\end{equation*}
The connection matrix $M \in M_{21 \times 21}(\QQ(t))$ has denominator 
$r = r_1 r_2 r_3$, with polynomials $r_1,r_2,r_3 \in \ZZ[t]$ that are 
irreducible of degree $16$, $104$ and $108$, respectively.  The matrix~$M$ 
has a simple pole at $\infty$ and the set of exponents is $\{1,2,3\}$ there.

\subsubsection{Precisions}

We take $\tau$ to be  $1 \in \FF_7$. We first compute $\delta=0$ and 
$N_{\Phi}=N_{\Phi}'=4$. Since we do not know the exponents of~$M$ at 
its finite poles, we cannot use Theorem~\ref{thm:KedlayaTuitman} at poles 
outside the residue disk at infinity. Applying Theorem~\ref{thm:Gerkmann} 
at all the zeros $z$ of $R$, using the bound 
$\ord_z(\Delta_k^{-1}) \geq -\ord_z(\det(\Delta_k))$, we find that we can 
take $s = \det(\Delta_2)\det(\Delta_3)\det(\Delta_4)^{67}$, which has 
degree~$25027$. However, most of the zeros of $R$ do not lie in the 
residue disk of a pole of~$M$ and hence the corresponding factors can 
be removed from~$s$, which decreases the degree of~$s$ to~$8833$.  
We can apply Theorem~\ref{thm:KedlayaTuitman} at~$\infty$ to find 
$\theta_{\infty}=-4$, so that we can take $K=8830$. We now compute 
the remaining $p$-adic precisions $N_{\Phi_0}=22$, $N_C=12$, $N_{C^{-1}}=14$, 
$N_M=30$, $N_C'=29$, $N_{C^{-1}}'=24$, $N_{\Phi_0}'=25$, and $N_M'=30$.

\subsubsection{Timings}

\begin{center}
\begin{tabular}{l l} \toprule
Computation     & P--T \\ \midrule
$M$             & 3.55m     \\
$\Phi_0$        & 0.02s     \\
$\Phi$          & 17.53m     \\
$Z(X_{\tau},T)$ & 6.7s     \\
Total           & 21.19m     \\ \bottomrule
\end{tabular}
\end{center}

\bigskip

As the final result of the computation we find that $7 \chi(T/7)$ is equal to 
\begin{multline*}
7T^{21} - 5T^{20} + 6T^{19} - 6T^{18} + 4T^{17} - 11T^{16} + 5T^{15} - 9T^{14} + 4T^{13} + 3T^{12} \\
+ 4T^{11} + 4T^{10} + 3T^9 + 4T^8 - 9T^7 + 5T^6 - 11T^5 + 4T^4 - 6T^3 + 6T^2 - 5T + 7.
\end{multline*}

\subsection{Larger primes}

We consider the family of quartic surfaces over $\ZZ$ given by the polynomial 
\begin{equation*}
P=x_0^4 + 2x_1^4 + 3x_2^4 + 4x_3^4 + t(x_0x_1x_2x_3 + 2 x_0^2 x_2^2 + 3x_0x_2^3)
\end{equation*}
which we consider as an element of $\ZZ_p[t][x_0,x_1,x_2,x_3]$ for various
primes $p$, most of which are much larger than those in the previous examples.
The connection matrix $M \in M_{21 \times 21}(\QQ(t))$ has denominator 
$r = r_1 r_2 r_3 r_4$, with polynomials $r_1,r_2,r_3,r_4 \in \ZZ[t]$ that are 
irreducible of degree $3$, $5$, $12$ and $16$, respectively. The set of
exponents at each of the zeros of $r$ is a subset of the set $\{-3/2,-1/2,0,1\}$,
and after changing basis by 
\[
W=\mbox{diag}(t^{-1},1,1,1,1,1,1,1,1,1,1,1,1,1,1,t^{-1},1,1,t^{-1},1,t)
\] 
as in Remark~\ref{rem:basischange}, the set of exponents  at $\infty$ is
$\{1,3/2,2,3\}$. We take $\tau$ to be $1 \in \FF_p$.

\subsubsection{Precisions}

For all primes $p$ that we consider, the zeros of $r$ are $p$-adic integers and different modulo $p$.
Therefore, we can apply Theorem~\ref{thm:KedlayaTuitman} everywhere. As an example, we determine the
precisions in the case $p=2^9+9$. We first compute $\delta=0$ and $N_{\Phi}=N'_{\Phi}=3$. We find
$\theta_z=1824$ at all zeros $z$ of $r$ and $\theta_{\infty}=4$, so that we can take $s=r^{1824}$ and $K=65668$.
We now compute the remaining $p$-adic precisions $N_{\Phi_0}=9$, $N_C=5$, $N_{C^{-1}}=7$, $N_M=14$, 
$N_C'=13$, $N_{C^{-1}}'=8$, $N_{\Phi_0}'=11$ and $N_M'=14$.

\subsubsection{Timings}

\begin{center}
\begin{tabular}{l l l l l l} \toprule
$p$             & $2^2+1$ & $2^3+3$ & $2^4+1$ & $2^5+5$ & $2^6+3$ 	\\ \midrule
Computation     & 	  &         &         &         &               \\ \midrule  
$M$             & 0.10s   & 0.10s   & 0.10s   & 0.10s  	& 0.10s	  	\\
$\Phi_0$        & 0.01s   & 0.00s   & 0.01s   & 0.01s 	& 0.00s	  	\\
$\Phi$          & 5.89s   & 9.19s   & 9.38s   & 20.06s 	& 42.18s       	\\
$Z(X_{\tau},T)$ & 0.07s   & 0.14s   & 0.08s   & 0.37s 	& 0.91s	  	\\
Total           & 6.07s   & 9.43s   & 9.57s   & 20.54s 	& 43.19s      	\\ \bottomrule
\end{tabular} 
\end{center}

\begin{center}
\begin{tabular}{l l l l l l} \toprule
$p$             & $2^7 + 3$ & $ 2^8+1$	& $2^9+9$  	& $2^{10}+7$ 	& $2^{11}+5$	\\ \midrule
Computation     & 	    &          	&          	&            	& \\ \midrule  
$M$             & 0.10s     & 0.10s	& 0.10s    	& 0.10s	 	& 0.10s \\
$\Phi_0$        & 0.00s     & 0.00s	& 0.00s       	& 0.00s	      	& 0.00s\\
$\Phi$          & 84.78s    & 166.68s	& 359.06s    	& 748.57s	& 1569.01s\\
$Z(X_{\tau},T)$ & 1.61s     & 3.62s	& 7.09s      	& 14.50s	& 32.03s\\
Total           & 86.49s    & 170.40s	& 366.25s    	& 763.18s	& 1601.15s\\ \bottomrule
\end{tabular} 
\end{center}

\bigskip

Our timings in this example confirm that the running time of the algorithm is (quasi)linear in $p$
and the memory usage also turns out to be (quasi)linear in $p$ as predicted. In the case $p=2^{11}+5$, the computation 
requires about 6 GB of memory. As an example, for the prime $p=2^{11}+5$ we find that $p \chi(T/p)$ is equal to
\begin{multline*}
2053 T^{21} - 4885 T^{20} + 6922 T^{19} - 7050 T^{18} + 1163 T^{17} + 7077 T^{16} - 12948 T^{15} \\ 
+ 12948 T^{14} - 9130 T^{13} + 3722 T^{12} + 128 T^{11} + 128 T^{10} + 3722 T^9 - 9130 T^8 \\ 
+ 12948 T^7 - 12948 T^6 + 7077 T^5 + 1163 T^4 - 7050 T^3 + 6922 T^2 - 4885 T + 2053.
\end{multline*}

\begin{rem}
In all of our timings, we have not included the time required to compute 
the $p$-adic and $t$-adic precisions, since we used MAGMA for this. Note, 
however, that we can readily determine $\det(\Delta_k)$ from the 
$LUP$~decomposition of~$\Delta_k$ and that the factorisation of the 
polynomials $r$, $R$ over $\QQ$ and $\FF_p$ was instantaneous in all 
cases. The only step that required noticeably more time was the 
computation of the exponents in the generic examples. For the family of 
generic quintic curves, this required $10s$, $47s$ and $1.5s$ at the 
zeros of $r_1$, the zeros of $r_2$ and at $\infty$, respectively. For the 
family of generic quartic surfaces, this required $13s$ and $59s$ at the 
zeros of~$r_1$ and at~$\infty$, respectively, but almost four hours for 
both the zeros of~$r_2$ and $r_3$.  Note, however, that in this last case 
we only use the exponents at~$\infty$ in the precision analysis and that
the other exponents are not needed for the result of the computation to
be provably correct.
\end{rem}

\phantomsection

\bibliographystyle{plainnat}
\bibliography{deformation}

\end{document}